\documentclass[11pt,a4paper]{amsart}

\usepackage{amsmath,amsfonts,amssymb,amsthm}
\usepackage{txfonts}
\usepackage{latexsym,bm,graphicx}
\usepackage{mathrsfs}
\usepackage{color}

\title[Optimal boundary regularity of  harmonic maps]{Optimal boundary regularity of harmonic maps from $RCD(K,N)$-spaces to $CAT(0)$-spaces}
\author{Hui-Chun Zhang}
\address{Department of Mathematics\\  Sun Yat-sen University\\ Guangzhou 510275\\ \newline E-mail address: zhanghc3@mail.sysu.edu.cn}
\author{Xi-Ping Zhu}
\address{Department of Mathematics\\  Sun Yat-sen University\\ Guangzhou 510275\\ \newline E-mail address: stszxp@mail.sysu.edu.cn}

 \newtheorem{theorem}{Theorem}[section]

\newtheorem{proposition}[theorem]{Proposition}
\newtheorem{lemma}[theorem]{Lemma}

\theoremstyle{definition}
\theoremstyle{remark}

\newtheorem{exam}[theorem]{Example}

\newtheorem{defn}[theorem]{Definition}
\newtheorem{remark}[theorem]{Remark}

\numberwithin{equation}{section}

\newcommand{\ls}{\leqslant}
\newcommand{\gs}{\geqslant}

\newcommand{\ip}[2]{\left<{#1},{#2}\right>}

\newcommand{\meas}{\mathfrak{m}}

\begin{document}


\begin{abstract}
  In 1983, Schoen-Uhlenbeck \cite{SU83} established boundary regularity for energy-minimizing maps between smooth manifolds with the Dirichlet boundary condition under the assumption that both the boundary and the data are of  $C^{2,\alpha}$. A natural problem is to study the qualitative boundary behavior of harmonic maps with rough boundary and/or non-smooth boundary data.

For the special case where $u$ is a harmonic function on a domain $\Omega\subset \mathbb R^n$, this problem has been extensively studied (see, for instance,  the monograph \cite{Kenig94},  the proceedings of ICM 2010 \cite{Tor10} and the recent work of Mourgoglou-Tolsa \cite{MT24}).  The $W^{1,p}$-regularity ($1<p<\infty$) has been well-established when $\partial\Omega$ is Lipschitz (or even more general) and the boundary data belongs to $W^{1,p}(\partial\Omega)$. However, for the endpoint case where the boundary data is Lipschitz continous,  as demonstrated by Hardy-Littlewood's classical examples \cite{HL32},  the gradient $|\nabla u|(x)$ may have logarithmic growth as $x$ approaches the boundary $\partial \Omega$ even if the boundary is smooth.

In this paper, we first establish a version of the Gauss-Green formula for bounded domains in $RCD(K, N)$ metric measure space. We then apply it to obtain the optimal boundary regularity of harmonic maps from  $RCD(K, N)$ metric measure spaces into $CAT(0)$ metric spaces. Our result is new even for harmonic functions on Lipschitz domains of Euclidean spaces.\\[3pt]

\end{abstract}

\maketitle

\section{Introduction}

In this paper we study the boundary regularity of harmonic maps between singular metric spaces with the prescribed Dirichlet boundary conditions.

\subsection{Boundary regularity of harmonic maps between smooth spaces}
In the setting of smooth spaces, Qing \cite{Qing93} proved that any weakly harmonic map from surfaces is smooth up to the boundary when both the boundary and the data are smooth. Schoen-Uhlenbeck \cite{SU83}  showed the following boundary regularity for energy-minimizing maps between smooth manifolds, assuming $C^{2,\alpha}$-regularity on both the boundary and the data.
\begin{theorem}[Schoen-Uhlenbeck \cite{SU83}]
\label{BR-thm-1.1}
Let $M$ be a compact manifold with $C^{2,\alpha}$ boundary $\partial M$. Suppose $u\in W^{1,2}(M,N)$ is an energy-minimizing map and satisfies $u(x)\in N_0$ a.e. for a compact subset $N_0$ in a smooth Riemannian manifold $N$. Suppose that $u_0\in C^{2,\alpha}(\partial M,N_0)$ and $u=u_0$ on $\partial M$. Then $u$ is $C^{2,\alpha}$ in a full neighborhood of $\partial M.$
\end{theorem}

A natural problem is to study {\emph{the qualitative boundary behavior of a harmonic map $u:\Omega\to N$ with rough boundary $\partial\Omega$ and/or non-smooth boundary data $u_0$}}.

For the special case when $u$ is a harmonic function on $\Omega\subset \mathbb R^n$, this problem has been widely researched.  For example, the Wiener criterion for $\partial \Omega$ ensures the continuity of harmonic functions up to the boundary when the boundary data is continuous.   Fabes-Jodeit-Riv\'ere \cite{FJR78}investigated harmonic functions on $C^1$-domains with $L^p(\partial\Omega)$-data. Jerison-Kenig \cite{JK81,JK82} considered harmonic functions in $W^{1,p}$-domains (or NTA-domains) and $L^p$-boundary data.    Shen \cite{Shen07} studied the $L^p$ boundary value problems on Lipschitz domains.      Mitrea-Mitrea-Yan \cite{MMY10} analyzed harmonic functions on $\Omega$ under the assumption that $\partial\Omega$  is Lipschitz and satisfies an exterior ball condition, and that the data $u_0$   is in $W^{1,p} $ for $1<p<\infty$. Recently, Mourgoglou-Tolsa \cite{MT24}  established the $W^{1,p}$-regularity ($1<p<\infty$) when the boundary measure of $\partial\Omega$ is Ahlfors-David regular (a more general condition than  Lipschitz regularity)  and the boundary data is in $W^{1,p}(\partial \Omega)$. For further references, we refer the readers to the monograph \cite{Kenig94}, the proceedings of ICM 2010 \cite{Tor10} and the references in \cite{MT24}.
On the other hand, for the endpoint case where the boundary data $u_0$ is Lipschitz continous,  Hile-Stanoyevitch \cite{HS99} (after  Hardy-Littlewood \cite{HL32}) proved that if $\partial \Omega\in C^2$ and $u_0$ is Lipschitz,  the gradient $|\nabla u|(x)$  grows logarithmically as $x$ approaches the boundary $\partial \Omega$ when  $\partial \Omega\in C^2$.  

The Wiener criterion has been extended from harmonic functions to harmonic maps from bounded Lipschitz domains, whose image lies in a convex ball \cite{ACM91}. Jost-Meier \cite{Jost-M83} have established the H\"older boundary regularity for minima of certain quadratic functionals, assuming $C^1$-continuity of the boundary $\partial \Omega$.

In this paper, we will prove the following gradient estimates of harmonic maps near the boundary. Indeed, it is a corollary of our main result --- Theorem \ref{BR-thm-1.8-main}.
\begin{theorem}\label{BR-thm-1.2}
 Let  $\Omega$ be a bounded domain in an $n$-dimensional Riemannian manifold  $M^n$ with boundary $\partial\Omega$, which is Lipschitz continuous and satisfies a uniformly exterior ball condition, $n\gs2$. Let $N$ be a Hadamard manifold (i.e.,  simply connected manifold with non-positive sectional curvature). Then 
for any $u_0\in Lip ({\overline\Omega},N)$ with Lipschitz constant $L>0$,
the  energy-minimizing map $u\in W^{1,2}(\Omega,N)$ with   $u=u_0$ on $\partial \Omega$ satisfies the following gradient estimate
  \begin{equation}\label{equation-1-2}
  	|\nabla u|(x)\ls c  L\cdot \ln\left(\frac{2e\cdot {\rm diam}(\Omega)}{d(x,\partial\Omega)}\right),\qquad \forall x\in \Omega,
\end{equation}
where the constant $c$ depends only on the domain $\Omega.$
\end{theorem}

\begin{remark}\label{BR-rem-1.3}
(1)  Theorem \ref{BR-thm-1.2} is new even for harmonic functions on Euclidean domains.  Recall that Hardy-Littlewood \cite{HL32}  established (\ref{equation-1-2}) for harmonic functions  $u$ on  the half-space $\Omega:=\{(x,x_n)\in\mathbb R^n|\ x_n>0, x\in \mathbb R^{n-1}\}$.  Hile-Stanoyevitch \cite{HS99} extended this estimate to domains with $C^2$ boundaries via the flattening of the boundaries. However, the flattening technique does not work for Lipschitz boundaries as in Theorem \ref{BR-thm-1.2}. Indeed, we will show that (\ref{equation-1-2}) remains valid under more general assumptions that the boundaries are perimetrically regular (see Definition \ref{BR-def-1.5}(2)). 

 (2)  The exterior ball condition in Theorem \ref{BR-thm-1.2} is nearly indispensable, as illustrated in the following simple examples. Let $a\in(1/2,1)$ and consider the cone
 $$C_a:=\{(\rho,\theta) |\ 0<\rho<+\infty,\ 0<\theta<\pi/a\} $$
with  the vertex $O=(0,0)$, where $(\rho, \theta)$ are the polar coordinates in $\mathbb R^2$. Then the domain $\Omega_a:=C_a\cap B_1(O)$ satisfies only an exterior cone condition at  $O$ with the cone angle $(2-1/a)\pi<\pi$. A direct calculation shows that the function $f_a:=\rho^{a}\sin(a\theta)$ is harmonic on $\Omega_a$ and has Lipschitz boundary data (indeed $f_a(x)=0$ for any $x\in\partial C_a\cap B_1(O)$). Moreover,
 $$|\nabla f_a|(x)\geqslant \left|\frac{\partial f_a}{\partial \rho}\right|(x)=a\rho^{a-1}=a\cdot d^{a-1}(x,\partial\Omega_a),\quad {\rm where}\ \ x=(\rho,\frac{\pi}{2a}).$$

 (3) The assumption that the target $N$ is a Hadamard manifold can be replaced by the assumption that $N$ is a manifold with curvature bounded from above by $\kappa>0$ and that the image of $u$ is contained in a ball of radius  $<\pi/(2\sqrt\kappa)$. See Remark \ref{BR-remark-6.2} for the precise statement for more general $CAT(\kappa)$ target spaces.
\end{remark}

\subsection{Boundary regularity of harmonic maps between singular spaces and the main result} From  Gromov-Schoen \cite{GS92},  there has been growing interest in developing a theory of harmonic maps between singular metric spaces (see, for example, \cite{KS93, Che95, Lin97, Jost94, Jost95, Jost97, DM08, DM10, EF01, Stu05,  Guo21, Hon-S23} and their references; the list is far away from complete).

To introduce our main result, let us first recall some notations.
Given a metric measure space $(X,d,\meas)$, the $RCD(K,N)$ is  a synthetic notion of lower Ricci bounds for  $(X,d,\meas)$, which  has been  developed in    \cite{Stu06a,Stu06b,LV09, AGS14a,AGS14b,Gig15,AMS16,EKS15}. The parameters $K\in\mathbb R$ and $N\in [1,+\infty]$ play the role of “Ricci curvature $\geqslant   K$ and dimension $\leqslant N$”.   We refer to \cite{Amb18} for the survey of geometric analysis on $RCD(K,N)$ spaces.
 Let $\Omega$ be a bounded open domain of  an $RCD(K,N)$ space $(X,d,\meas)$ for some $K\in\mathbb R$ and $N\in [1,+\infty)$ such that $\meas(X\backslash \Omega)>0$, and let $(Y,d_Y)$ be a complete metric space.
 A Borel measurable map $u:\ \Omega\to Y$ is said to be in the space $L^2(\Omega,Y)$ if it has a separable range and, for some (hence, for all) $P\in Y$,
$$\int_\Omega d^2_Y\big(u(x),P\big){\rm d}\meas(x)<\infty.$$
The Sobolev energy of  a map   $u\in L^2(\Omega,Y) $ is defined \cite{KS93,KS03,GT21b} by
\begin{equation}
	\label{BR-equ-1.1}E_2^{\Omega}(u):=\sup_{\phi\in C_0(\Omega),\ 0\ls \phi\ls 1}\left(\limsup_{r\to 0^+} \int_\Omega\phi\cdot {\rm ks}^2_{2,r}[u,\Omega]{\rm d}\meas\right),
\end{equation}	
where
 $$  {\rm ks}_{2,r}[u,\Omega](x):=
\begin{cases}
\left(\frac{1}{\meas(B_r(x))}\int_{B_r(x)} \frac{d_Y^2(u(x), u(y))}{r^2}{\rm d}\meas(y)  \right)^{1/2}, &\ {\rm if}\ \ B_r(x)\subset\Omega,\\
0,& \ \  {\rm otherwise}.
\end{cases}
$$
A map $u\in L^2(\Omega,Y) $ is called in $W^{1,2}(\Omega,Y)$ if $E_2^\Omega(u)<+\infty.$
 \begin{defn}\label{BR-def-1.2}
A map $u\in W^{1,2}(\Omega, Y)$ is called a {\emph{harmonic map}} if it solves the Dirichlet problem for some prescribed  data $u_0 \in W^{1,2}(\Omega, Y)$. Precisely,
given a map $u_0\in W^{1,2}(\Omega,Y)$,  we have
  $$E_2^{\Omega}(u)= \min \left\{ E_2^{\Omega}(v)\big |\ v\in W^{1,2}(\Omega,Y) \ {\rm such \ that }\ d_Y(v(x),u_0(x)) \in W^{1,2}_0(\Omega)\right\}.$$
\end{defn}

If $(Y, d_Y)$ is a $CAT(0)$ space (i.e., a metric space with global non-positively curvature in the sense of Alexandrov), then for any  $u_0\in W^{1,2}(\Omega,Y)$ the harmonic map $u$ with the boundary data $u_0$ exists uniquely (see \cite{GT21b}).

The interior regularity of such harmonic maps has been considered by many researchers  (see, for example, \cite{Che95, Lin97, Jost97, DM08, DM10, EF01, Fug03, Guo21, ZZ18, ZZZ19, MS22+, Gig23}). In 1997, Jost \cite{Jost97} and Lin \cite{Lin97} proved that every harmonic map from a finite-dimensional Alexandrov space with curvature bounded from below to a $CAT(0)$ space is locally H\"older continuous. Lin \cite{Lin97} and Jost \cite{Jost98} proposed an open problem:
Can the H\"older continuity be improved to Lipschitz continuity? The authors gave a complete resolution to it in \cite{ZZ18}.
 Recently, Mondino-Semola \cite{MS22+} and   Gigli \cite{Gig23} have independently proved a very general interior regularity result: any harmonic map from an $RCD(K, N)$ space to a $CAT(0)$ space is locally Lipschitz continuous, where $K\in\mathbb R$ and $N\in[1,\infty)$.  Assimos-Gui-Jost \cite{Ass-GJ24} have proved the local Lipschitz continuity for sub-elliptic harmonic maps from $n$-dimensional Heisenberg groups into $CAT(0)$-spaces.

In this present paper, we will focus on {\emph{the qualitative boundary behavior of harmonic maps between singular metric spaces with   Lipschitz data (of course, the boundary is non-smooth)}}.

 Serbinowski \cite{Ser94} considered first the boundary regularity for harmonic maps from a smooth domain to a $CAT(0)$ metric space. He \cite{Ser94} proved that such a harmonic map is globally H\"older continuous with any H\"older index $\alpha\in(0,1)$, assuming that the boundary is of $C^1$ and that the data is Lipschitz continuous.

 To deal with the case where the domain is non-smooth, we recall the following two conditions for the boundary.
 \begin{defn} \label{BR-def-1.5}
Let $\Omega\subset X$  be an open domain with $X\setminus \Omega\not=\emptyset.$

(1)  $\Omega$ is said to satisfy a  {\emph{uniformly exterior ball condition with radius $R_{\rm ext}\in(0,1)$}} if for each $x_0\in \partial\Omega$ there exists a point $y_0\in X\setminus\Omega$ such that
\begin{equation}\label{BR-equ-1.2}
d(y_0,x_0)=R_{\rm ext}\quad {\rm and}\quad \overline{B_{R_{\rm ext}}(y_0)}\cap \Omega=\emptyset.
\end{equation}

(2)   $\partial\Omega$ is said to be  {\emph{perimetrically regular}}  if it is a set of finite perimeter and  there exists a constant  $C_0>0$ such that
\begin{equation}\label{BR-equ-1-3}
Per_{\Omega}(B_r(x_0))\ls C_0\frac{\meas(B_r(x_0))}{r},\qquad \forall x_0\in\partial\Omega,\ \ \forall r\in (0,{\rm diam}(\Omega)).
\end{equation}
Please see \cite{Mir03, ABS19} (or Section 2.3) for the definition of sets of finite perimeter and the perimeter measure $Per_\Omega$.
\end{defn}

\begin{remark}\label{BR-remark-1.6}
 (1) Every bounded domain $\Omega\subset \mathbb R^N$ with a Lipschitz boundary is perimetrically regular. Furthermore,  if the $(N-1)$-dimensional Hausdorff measure $\mathscr H^{N-1}$ is Ahlfors-David regular on $\partial \Omega$  then $\Omega$ is also perimetrically regular (for the definition of Ahlfors-David regularity, see, for example, \cite{MT24}).

(2) There are a lot of bounded domains $\Omega\subset \mathbb R^N$ such that they are perimetrically regular and satisfy a uniformly exterior ball condition, but they do not have a Lipschitz boundary.  An example is given by
$$\Omega:=\left\{(x,y)\in\mathbb R^2\big|\ 0< x<1,\ \ \sqrt{1-x^2}<y<1\right\}.$$

\end{remark}

  In the previous work of the authors \cite{ZZ24}, we have extended Serbinowski's result to the case where both the domain space and the target space are non-smooth.

  \begin{theorem}[Zhang-Zhu \cite{ZZ24}]\label{BR-thm-1.7}
 Let $(X,d,\meas)$ be an $RCD(K,N)$ space with $K\in\mathbb R$ and $N\in(1,+\infty)$, and let $\Omega\subset X$  be a bounded domain such that $\meas(X\setminus \Omega)>0$. Let $(Y,d_Y)$ be a $CAT(0)$ space. Suppose that $u\in W^{1,2}(\Omega,Y)$  is a harmonic map  with data $u_0$. Suppose that    $ \Omega$ satisfies a uniformly exterior ball condition with radius $R_{\rm ext}\in(0,1)$ and the data $u_0\in Lip(\overline{\Omega},Y).$
 Then $u\in C^{1-\epsilon}(\overline{\Omega},Y)$ for every $\epsilon>0$.
  \end{theorem}
  \noindent Here and in the sequel,  a map $v\in C^\alpha(\overline\Omega,Y)$, $0<\alpha<1$,  means
 $$\sup_{x,y\in\overline\Omega,\ x\not=y}\frac{d_Y(v(x),v(y))}{d^\alpha(x,y)} <+\infty.$$
  A map $v\in Lip(\overline\Omega,Y)$ means $ d_Y(v(x),v(y))\ls L\cdot d(x,y)$ for some $L>0$ and for all $x,y\in\overline\Omega$; such a constant $L$ is called a Lipschitz constant of $v$.

Our main result in this paper is the following optimal boundary gradient estimates of $u$ on  $RCD(K, N)$ spaces.

\begin{theorem}\label{BR-thm-1.8-main}
 Let $(X,d,\meas)$ be an $RCD(K,N)$ space with $N>1$,  and let $\Omega\subset X$  be a bounded domain such that $\meas(X\setminus \Omega)>0$. Let $(Y,d_Y)$ be a $CAT(0)$ space. Suppose that $\partial\Omega$ is perimetrically regular, and that $\Omega$ satisfies a uniformly exterior ball condition with radius $R_{\rm ext}\in(0,1)$.

 Given any $u_0\in Lip ({\overline\Omega},Y)$ with Lipschitz constant $L>0$,
 let $u\in W^{1,2}(\Omega,Y)$ be the harmonic map with the boundary data  $u_0$. Then it holds
\begin{equation}
	\label{BR-equ-1.4}
	{\rm Lip}u(x):=\limsup_{y\to x,\ y\not=x}\frac{d_Y(u(x),u(y))}{d(x,y)}\ls c_{1} L\cdot \ln\left(\frac{2e\cdot {\rm diam}(\Omega)}{d(x,\partial\Omega)}\right),\quad \forall x\in\Omega,
\end{equation}
where  the constant $c_1$ depends only on $K$, $N$, $ {\rm diam}(\Omega), \meas(\Omega), Per_\Omega(X),  R_{\rm ext}$ and $C_0$ in \eqref{BR-equ-1-3}.
\end{theorem}

From the formula of the Poisson kernel on the unit ball in $\mathbb R^N$ and the calculation in \cite{HL32},  it is clear that the estimate (\ref{BR-equ-1.4}) is optimal. Precisely, we have the following example.
  \begin{exam}[Hardy-Littlewood]  \label{BR-exam-1.9}
Let  $B^2_{1}(0)=\{(r,\theta)|\ 0\ls r<1, \theta\in(-\pi,\pi]  \}$ be the unit ball in $ \mathbb R^2$, and let $f_0$ be defined on $S_1=\partial B^2_1(0)$ by
$$f_0(\theta)=|\sin\theta|  \  \ {\rm if}\ \  |\theta|\ls \pi/2 \qquad {\rm and}\qquad   f_0(\theta)=1\ \ {\rm otherwise}.$$
Then the harmonic function on $B^2_{1}(0)$ with the data  $f_0$ is uniquely determined by the Poisson integral formula,
$$f(r,\theta_0)= \frac{1-r^2}{2\pi}\int_{ S^1} \frac{ f_0(\theta)}{|(r,\theta_0),(1,\theta)|_{\mathbb R^2}^2 }{\rm d}\theta = \frac{1-r^2}{2\pi}\int_{S^1} \frac{f_0(\theta)}{ r^2+1-2r\cos(\theta-\theta_0)   }{\rm d}\theta.$$
 By directly calculating, we have
 \begin{equation*}
 \begin{split}
2\pi\cdot \frac{\partial f(r, \theta_0:=0)}{\partial r}&= \int_{|\theta|\ls\pi/2}\frac{-4r +(2r^2+2)\cos\theta }{(r^2+1-2r\cos \theta )^2} |\sin\theta|{\rm d}\theta\\
  &\quad + \int_{|\theta|\gs\pi/2}\frac{-4r +(2r^2+2)\cos\theta }{(r^2+1-2r\cos \theta )^2}  {\rm d}\theta \\
 &:= 2\big( I_1(r)+I_2(r)\big)
    \end{split}
 \end{equation*}
 for all $r<1$,
where
\begin{equation*}
\begin{split}
I_1(r)&= \int_{|\theta|\ls\pi/2}\frac{-2r +(r^2+ 1)\cos\theta}{(r^2+1-2r\cos \theta )^2} |\sin\theta|{\rm d}\theta\\
 &= -2\int_0^1\frac{-2r+(r^2+1)  s}{(r^2+1-2rs)^2}{\rm d}s\qquad \qquad ({\rm by}\ \ s=\cos\theta)\\
 &=\frac{(r+1)^2}{r(r^2+1)}-\frac{1+r^2}{r^2}\ln\left(1-r\right)+\frac{r^2+1}{2r^2}\ln(1+r^2)\\
 &\sim 2\ln  \left( \frac{1}{1-r} \right),\ \qquad  ({\rm as}\ \ r\to1^-)
\end{split}
\end{equation*}
and, noticing that $r^2+1-2r\cos\theta\gs1$ whenever  $|\theta|\gs \pi/2$,
 $$|I_2(r)|\ls  \int_{|\theta|\gs\pi/2}\big| -2r +( r^2+1)\cos\theta  \big| {\rm d}\theta\ls 4\pi.$$
Therefore, we conclude
\begin{equation*}
|\nabla f|(r,0) \gs \left|\frac{\partial f(r,0)}{\partial r}\right| \gs  \ln \left(  \frac{1}{1-r} \right)=\ln\left(\frac{1}{|(r,0),
\partial B_2(0)|_{\mathbb R^2}}\right),
 \end{equation*}
as $r\to 1^-$.\end{exam}

\subsection{Outline of the proofs}
 We now give an imprecise outline of the proof of Theorem \ref{BR-thm-1.8-main}. To simplify, we begin with the toy model case when $u$ is a harmonic function (i.e., $N=\mathbb R$) on an Euclidean domain with Lipschitz boundary data $u_0$.

 Let $\Omega\subset \mathbb R^n$ be a set of locally finite perimeter that is perimetrically regular and satisfies a uniform exterior ball condition. Let $F:\mathbb R^n\to \mathbb R^n$ be a Lipschitz (or BV) vector field. De Giorgi's structure theorem implies that the classical Gauss-Green formula is also true for $(\Omega, F)$ (see De Giorgi \cite{DeG54}, Federer \cite{Fed58}, and Ambrosio-Fusco-Pallara \cite{AFP00}), that is, it holds
 \begin{equation}\label{equation-1-7}
\int_\Omega \phi  div(F) {\rm d}x+\int_\Omega \ip{F}{\nabla \phi}{\rm d}x=-\int_{\mathscr F \Omega}\phi F\cdot \nu_\Omega{\rm d}\mathscr H^{n-1},\quad \forall \phi\in Lip_0(\Omega),	
\end{equation}
where ${\nu_\Omega}$ is the measure-theoretic interior unit normal vector to $\partial \Omega$ and $\mathscr F\Omega$ is the reduced boundary.   Denote by  $G^\Omega(x,y)$  the Green function of the Dirichlet Laplacian on $\Omega$. Fix any $x\in\Omega$.
If we priorly assume that   $\nabla G^\Omega(x,\cdot)\in Lip(\Omega,\mathbb R^n)$,  then $\nabla G^\Omega(x,\cdot)$ can be naturally extended to a Lipschitz vector field on $\mathbb R^n$. By applying \eqref{equation-1-7} to $F:=\nabla G^\Omega$ and $\phi:=u$, we conclude
\begin{equation}\label{BR-equ-1.5}
|u(x)| \ls \int_{\mathscr F  \Omega}|u_0(y)| \cdot|\nabla G^\Omega(x,y)| {\rm d}\mathscr H^{n-1}(y).
\end{equation}
Since $\Omega$ satisfies a uniform exterior ball condition, Gr\"uter-Widman \cite{GW82} proved
\begin{equation}\label{equation-1-9}	|\nabla G^\Omega(x,\cdot)|(y)\ls C_n\frac{{d}(x,\partial \Omega)}{|x-y|^n},\quad \forall y\in\Omega.
\end{equation}
Let $x_0\in\partial\Omega$ be a point such that $|x_0-x|={d}(x,\partial\Omega)$.
 The minimizing of $x_0$ yields that for any $y\in \partial\Omega$,
 $$|x-y|\gs \frac{2}{3}|x-x_0|+\frac{1}{3}|x-y|\gs \frac{1}{3}|x-x_0|+\frac{1}{3}|x_0-y|.$$
 Substituting this into (\ref{equation-1-9}), we get
\begin{equation}\label{equation-1-10}
|\nabla G^\Omega(x,\cdot)|(y) \ls C'_n    \frac{{d}(x,\partial \Omega)}{|x-x_0|^n+|x_0-y|^n},\quad \forall y\in\Omega.
\end{equation}
This implies the desired estimate (\ref{equation-1-2}), which we will illustrate as follows.  By replacing $u_0$ by $u_0-u_0(x_0)$, we can assume $u_0(x_0)=0$. The assumption  $u_0\in Lip(\overline\Omega)$ implies
$$|u_0(y)|\ls L|x_0-y|,\quad \forall \ y\in \partial\Omega,$$
 where $L$ is a Lipschitz constant of $u_0$.
Hence, by combining with \eqref{BR-equ-1.5}, \eqref{equation-1-10}, and letting $r_j:=2^{-j}{\rm diam}(\Omega)$, $B_j:=B_{r_j}(x_0)$, we have
\begin{equation}
\begin{split}
	\label{equation-1-11}
\frac{|u(x)|}{{d}(x,\partial \Omega)}& \ls C_n  \sum_{j=0}^\infty \int_{\mathscr F  \Omega\cap  ( B_j\setminus B_{j+1})} \frac{L |x_0-y|}{|x-x_0|^n+|x_0-y|^n} {\rm d}\mathscr H^{n-1}(y)\\
&\ls C'_nL\sum_{j=0}^\infty\frac{r_j}{|x-x_0|^n+r_j^n}\mathscr H^{n-1}\big(\mathscr F    \Omega\cap (B_j\setminus B_{j+1})\big)\\
&\ls C_n'L\sum_{j=0}^\infty\frac{r_j^n}{|x-x_0|^n+r_j^n}\qquad\ \ (by\ \ \eqref{BR-equ-1-3})\\
&\ls C_n''L\cdot\ln\left(\frac{2{\rm diam}(\Omega)}{|x-x_0|}\right),
\end{split}
\end{equation}
which implies the desired gradient estimate \eqref{equation-1-2}.

The primary issue with this outline is that there is no reason we can assume that $\nabla G^\Omega$ is Lipschitz continuous (or BV). Therefore, we must extend the Gauss-Green formula for more general non-smooth vector fields. A vector field $F\in L^\infty(U,\mathbb R^n)$, $U\subset \mathbb R^n$ open,  is called an {\emph{essentially bounded divergence-measure field}}, denoted by $F\in \mathcal{DM}^\infty(U)$, if its distributional divergence is a signed Radon measure with finite total variation on $U$.  Such fields were first introduced by Anzellotti in \cite{Anz83}. An important improvement of \eqref{equation-1-7}  was given by Chen-Frid \cite{CF99}, Chen-Torres-Ziemer \cite{CTZ09}, and Comi-Rayne \cite{CP20}, where they   established  the following weak form of the Gauss-Green formula:
if $F\in \mathcal{DM}^\infty(U)$ and  $\Omega\Subset U$ is a set of finite perimeter, then there exists a function $(F\cdot \nu_\Omega)_{\rm in} \in L^\infty(\mathscr F\Omega,\mathscr H^{n-1})$ on the reduced boundary $\mathscr F\Omega$ such that  for any $\phi\in Lip_0(U)$,	it holds
\begin{equation}\label{equation-1-12}
	\ip{F\cdot\nu}{\phi}_{\partial \Omega}:=\int_{\Omega^{(1)}}\phi{\rm d}  div(F)+\int_\Omega\ip{F}{\nabla \phi}{\rm d}x =-\int_{\mathscr F\Omega}\phi (F\cdot \nu_\Omega)_{\rm in}{\rm d}\mathscr H^{n-1}
	\end{equation}
and	(see \cite[Remark 4.3]{CP20})
\begin{equation}
	\label{equation-1-13}
\|(F\cdot\nu_\Omega)_{\rm in}\|_{L^\infty(\mathscr F\Omega;\mathscr H^{n-1})}\ls \inf_{\varepsilon>0} \|F\|_{L^\infty\left(\{x\in \Omega:\ {d}(x,\partial\Omega)<\varepsilon\}\right)}   \end{equation}
  where $\Omega^{(1)}$ is all points where $\Omega$ has density 1, and $\ip{F\cdot\nu}{\cdot}_{\partial \Omega}$ is called the {\emph{(inner) normal trace}} of $F$ on $\partial\Omega$. By Riesz representation theorem, $\ip{F\cdot\nu}{\cdot}_{\partial \Omega}$ is a signed Radon measure on $\mathscr F\Omega$ and $\ip{F\cdot\nu}{\cdot}_{\partial \Omega}=(F\cdot\nu_\Omega)_{\rm in}\cdot\mathscr H^{n-1}|_{\mathscr F\Omega}$.

When one wants to replace (\ref{equation-1-7}) with (\ref{equation-1-12})--(\ref{equation-1-13}) to estimate $|u(x)|$, the following two main dificulties arise:
   \begin{itemize}
\item[(i)]
 $F$ must be extended to be a vector field $\tilde F\in \mathcal{DM}^\infty(U)$ for some open set $U\Supset\Omega$,
 \item[(ii)] the $L^\infty$-estimate (\ref{equation-1-13}) is not enough to conclude a logarithmic growth like $(\ref{BR-def-1.2})$. In fact, even for the case $\Omega:=B_1(0^n)$, it is well-known \cite[Section 2.5]{GT01} that
$$\max_{y\in\partial B_1(0^n)}|\nabla G^{B_1(0^n)}(x,\cdot)|(y) \approx_n \frac{1}{{d}^{n-1}(x,\partial B_1(0^n))},$$
where $A\approx_n B$ means $c_1\ls A/B\ls c_2$ for two positive numbers $c_1,c_2$ depending only on $n$. By combining with (\ref{equation-1-13}), one has
$$\|(\nabla G^\Omega\cdot\nu_\Omega)_{\rm in}\|_{L^\infty(\mathscr F\Omega;\mathscr H^{n-1})}\ls \frac{c_n}{{  d}^{n-1}(x,\partial B_1(0^n))},$$
 which has lost the {\emph{main term}} $|x_0-y|^n$ in the denominator of the right-hand side of (\ref{equation-1-10}).
\end{itemize}
Nowadays, the difficulty (i) has been resolved by Chen-Li-Torres \cite[Theorem 4.2]{CLT20}, where they proved that if $\Omega$ satisfies
$$\mathscr H^{n-1}(\partial\Omega\setminus\Omega^{(0)})<+\infty,$$
then any $F\in  \mathcal{DM}^\infty(\Omega)$ can be extended to be a vector field in $\mathcal{DM}^\infty(\mathbb R^n),$ where $\Omega^{(0)}$ is all points where $\Omega$ has density 0.

To overcome the difficulty (ii), one needs to establish a version of the Gauss-Green formula with a pointwise upper bound of $|(F,\nu_\Omega)_{\rm in}|$ such that the term  $|x_0-y|^n$ in (\ref{equation-1-10}) can be preserved.
Remark that in \cite[Theorem 7.2]{CTZ09} and \cite[Theorem 3.7]{CP20}, it was proved that if $F\in   C(U)\cap \mathcal{DM}^\infty(U)$ for some open set $U$ and if $\Omega\Subset U$ is a set of finite perimeter, then  $(F\cdot\nu_\Omega)_{\rm in}$ coincides $\mathscr H^{n-1}$-a.e. with the classical inner product of $F$ and $\nu_\Omega$. Hence, in this case,  the following pointwise estimate holds
\begin{equation}
	\label{equation-1-14}
	|(F\cdot\nu_\Omega)_{\rm in}|\ls  |F|,\quad    \mathscr H^{n-1}{\rm-a.e.\ in }\ \mathscr F\Omega.
		\end{equation}
However, the assumption $F\in C(U)$ is still too strict, since there is still no reason we can assume that $\nabla G^\Omega$ is continuous on $\overline\Omega\subset U$.

The main observation in the present paper is that, assuming that $F$ is a gradient field of a Lipschitz function and $\partial\Omega\cap\Omega^{(1)}=\emptyset$, we can prove a version of Gauss-Green formula, which relaxes the assumption $F\in C(U)$ in \cite{CTZ09, CP20} to that there exists a upper semi-continuous $h$ on $\overline\Omega$ such that $|F|\ls h$ almost all points in $\Omega$, and obtains the conclusion that a pointwise estimate (\ref{equation-1-14}), up to a multiple constant, still holds. The precise and quantitative statement will be given in Proposition \ref{BR-prop-1.12} below (in the more general the $RCD$ setting).\\

 We now go back to the $RCD$ setting and give a sketch of the proof of Theorem \ref{BR-thm-1.8-main}. Similar to the above, the proof contains three main steps:
 \begin{itemize}
\item[Step 1:] to obtain a gradient estimate of the Green function of the Dirichlet Laplacian on $\Omega$.
\item[Step 2:] to establish an adequate version of the Gauss-Green formula (see Proposition \ref{BR-prop-1.12}). And then, as above, we can estimate the growth of harmonic functions near the boundary $\partial\Omega$.
\item[Step 3:] to finish the proof, by using the harmonic functions to control the growth of harmonic maps near the boundary.
 \end{itemize}

To get the estimate for the Dirichlet Green functions, we first establish the following boundary behavior of Dirichlet heat kernels on $RCD(K, N)$ spaces, which is of independent interest.

\begin{theorem}\label{BR-thm-1.10}

Let $\Omega$ be bounded and satisfy a uniform exterior ball condition with radius $R_{\rm ext}\in(0,1)$. Then the Dirichlet heat kernel $p^\Omega_t(x,y)$ has the following upper bounds: for any $x,y\in \Omega$ and any $  T\gs R^2_{\rm ext}$ it holds
\begin{equation}
	\label{BR-equ-1.6}
	p^\Omega_t(x,y)\ls \frac{d(x,\partial\Omega)\cdot d(y,\partial\Omega)}{t}\cdot\frac{ c_{T}}{\meas(B_{\sqrt t}(y))} \exp\left(-\frac{d^2(x,y)}{c_T \cdot t}  \right)
\end{equation}
for all $t\in(0,T)$, where the constant $c_T$ depends only on $T, N,K$ and $ R_{\rm ext}$.
\end{theorem}
Since a Green function is not in $W^{1,2}$ in general, (in particular, its gradient field is not in $\mathcal{DM}^\infty(\Omega)$), we will combine the Li-Yau estimate and Theorem \ref{BR-thm-1.10} to get a uniform estimate for a family of approximating Green functions.

\begin{remark}\label{BR-rem-1.11}
  (1) The boundary behaviors of Dirichlet heat kernels on a $C^{1,1}$ domain of a smooth Riemannian manifold have been widely studied, see, for example,  in \cite{Sera22, Zhang02,  Wang98, Hui92, Dav90} and others.

(2) The proof of \ref{BR-thm-1.10} is based on barrier functions. In the Euclidean setting \cite{GW82, Hui92},  the barrier functions are given by the boundary Schauder estimates. Here, we will construct the barrier functions via the Laplacian comparison theorem in \cite{Gig15}.
\end{remark}

 Recently, the weak version of the Gauss-Green formula (\ref{equation-1-12})--(\ref{equation-1-13}) has been extended to domains on $RCD$ metric measure spaces (see \cite[Theorem 1.6]{BPS23} and \cite{BCM22, BPS23-jems, GM24}). We also mention the related work of  Brena-Gigli for this problem (see \cite[Theorem 4.13]{BG24}).

As we have illustrated above, the  (\ref{equation-1-12})--(\ref{equation-1-13}) is not enough to establish the desired estimate (\ref{BR-equ-1.6}). We need the following version of the Gauss-Green formula:
\begin{proposition}	
\label{BR-prop-1.12}Let $(X,d,\meas)$ be an $RCD(K,N)$ space for some $K\in\mathbb R$ and $N\in(1,\infty)$.
 Let  $\Omega\subset X$ be a bounded and open set of finite perimeter. Assume that $$\liminf_{r\to0}\frac{\meas\big(B_r(x_0)\setminus\Omega \big)}{\meas\big(B_r(x_0)\big)} \gs \gamma,\quad \forall x_0\in \partial\Omega,$$
  for some $\gamma\in (0,1)$. 	
Let $f\in W^{1,2}(\Omega)\cap C(\overline\Omega)$, $f\gs0$.  Suppose that $g\in W^{1,2}(\Omega)$ such that  ${\bf \Delta }g  $ is a finite Radon measure on $\Omega$ and ${\bf \Delta}g\ls a\meas$ for some function $a(x)\in L^1(\Omega)$. Assume $|\nabla g|(x)\ls h(x)$ $\meas$-a.e. $x\in \Omega$ for some upper semicontinuous function $h$ on $\overline\Omega$. Then it holds
 \begin{equation}\label{BR-equ-1.16}
 	-\int_\Omega fa {\rm d}\meas\ls \int_\Omega\ip{\nabla f}{\nabla g}{\rm d}\meas+\frac{c}{\gamma}\int_{\partial\Omega}fh{\rm d}Per_\Omega,
 \end{equation}
 where the constant $c$ depends only on $N, K$ and ${\rm diam}(\Omega)$.	
 \end{proposition}

 This proposition is the core of the techniques of this paper.  Its proof is based on a family of approximations from the interior of $\Omega$ via the discrete convolutions of $\chi_\Omega$ (c.f. \cite{KKST14, KLLS19}).

 \begin{remark}
 	One benefit of the above Gauss-Green formula (\ref{BR-equ-1.16}) is that we don't need to have an extension of the (informal) vector field $\nabla g$. It is different from the Euclidean case that our argument does not provide a criterion for the extension of a vector field in the sense of $\mathcal{DM}^\infty$. It is still interesting to ask how to extend the criterion of Chen-Li-Torres \cite{CLT20} to the $RCD$ setting.
 \end{remark}

 We will organize this paper as follows.   In Section 2, we will provide some necessary materials on  $RCD(K, N)$ spaces and sets of finite perimeter. In Section 3, we will provide the optimal upper bounds for Dirichlet heat kernels and Green functions, including the proof of Theorem \ref{BR-thm-1.10}. In Section 4, we will prove a version of the Gauss-Green formula on the $RCD$ setting for domains with one-sided approximations. In Section 5, we will give the boundary behaviors of nonnegative subharmonic functions. Section 5 is devoted to the proof of the main Theorem \ref{BR-thm-1.8-main}.  \\

{\bf Acknowledge} The second author was partially supported by National Key R\&D Program of China (No. 2022YFA1005400) and NSFC 12271530. The first author was partially supported by NSFC 12025109.

\section{Preliminaries and notations}

 Let $(X,d,\meas)$ be a metric measure space, i.e., $(X,d)$ is a complete and separable metric space equipped with a non-negative Borel measure which is finite on any ball $B_r(x)\subset X$ and ${\rm supp}(\meas)=X$.
 We will denote by $B_r(x_0)$ the open ball in $X$. Given an open domain $\Omega\subset X$, we denote by $Lip(\Omega)$ (resp.  $C_0(\Omega)$, $Lip_0(\Omega)$, $Lip_{\rm loc}(\Omega)$) the space of Lipschitz continuous (resp. continuous with compact support, Lipschitz continuous with compact support, locally Lipschitz continuous) functions on $\Omega$.
  We denote by $L^p(E):=L^p(E, \meas)$ for any $p\in[1,+\infty]$ and any $\meas$-measurable subset $E\subset X$. For any $f\in L^1(E)$, we set
\begin{equation*}
	f_E:=\fint_Ef{\rm d}\meas:=\frac{1}{\meas(E)}\int_Ef{\rm d}\meas.
\end{equation*}

We will use $c_1, c_2, C, \cdots$ to denote  generic  positive constants depending only on $K$, $N$, $ {\rm diam}(\Omega), \meas(\Omega)$ and $R_{\rm ext}$ (see Definition \ref{BR-def-1.5}); they may differ from line to line.

 \subsection{Basic calculus on $RCD(K,N)$ spaces}

 Let $f\in C(X)$, the \emph{pointwise Lipschitz constant} or \emph{slope}, ${\rm Lip}f: X\to [0,+\infty]$, is given by
 \begin{equation*}
 {\rm Lip}f(x):= \limsup_{y\to x}\frac{|f(y)-f(x)|}{d(x,y)}     \end{equation*}
 if $x$ is not isolated, and ${\rm Lip}f(x)=0$ otherwise.
 The \emph{Cheeger energy} ${\rm Ch}: L^2(X)\to [0,+\infty]$ is
 $${\rm Ch}(f):=\inf\Big\{\liminf_{j\to +\infty}\int_X({\rm Lip}f_j)^2{\rm d}\meas:\ f_j\in Lip_{\rm loc}(X)\cap L^2(X),\ \ f_j\overset{L^2}{\to }f\Big\}.$$
The Sobolev space
$$W^{1,2}(X)=W^{1,2}(X,d,\meas):=\{f\in L^2(X)|\ {\rm Ch}(f)<\infty\}$$
with the norm
\begin{equation*}
    \|f\|^2_{W^{1,2}(X)}:=\|f\|^2_{L^2(X)}+{\rm Ch}(f).
\end{equation*}
For any $f\in W^{1,2}(X)$, it was showed \cite{Che99,AGS14a} that there exists a  {\it minimal weak upper gradient} $|\nabla f|$ such that
 $${\rm Ch}(f)=\frac{1}{2}\int_X|\nabla f|^2(x){\rm d}\meas(x),$$
 and for any $f\in Lip(X) $ that $|\nabla f|(x)={\rm Lip}f(x)$  $\meas-$a.e.   $x\in X$.

A metric measure space $(X,d,\meas)$ is called \emph{infinitesimally Hilbertian} if the Sobolev space $W^{1,2}(X)$ is a Hilbert space.  For an infinitesimally Hilbertian space $(X,d,\meas)$, it was proved \cite{Gig15} that  for any $f,g\in W^{1,2}(X)$ the limit
  \begin{equation}\label{BR-equ-2.1}
  \ip{\nabla f}{\nabla g}:=\lim_{\epsilon\to0^+}\frac{|\nabla (f+\epsilon g)|^2-|\nabla f|^2}{2\epsilon}
  \end{equation}
    exists and is in $L^1(X)$. This provides a canonical Dirichlet form
 \begin{equation*}
  	\mathscr E(f,g):=\int_X\ip{\nabla f}{\nabla g}{\rm d}\meas,\qquad \forall\ f,g\in W^{1,2}(X).
 \end{equation*}
     We denote by $\Delta$   the infinitesimal generator  of $\mathscr E$ with domain $D(\Delta)$  and by   $\{H_t:=e^{t\Delta}\}_{t\gs0}$ the corresponding  semi-group (heat flow).

Several equivalent definitions exist for the \emph{Riemannian curvature-dimension condition} on metric measure spaces \cite{EKS15, AMS16, AGS14b}.
\begin{defn}\label{BR-def-2.1}
Letting $K\in\mathbb R$ and $N\in[1,+\infty)$, a metric measure space $(X, d, \meas)$ is called an $RCD(K, N)$-space if the following conditions hold:
   \begin{enumerate}
       \item  It is infinitesimally Hilbertian.
       \item  For any $f\in D(\Delta )$ with $\Delta f\in W^{1,2}(X)$, it holds  the weak Bochner inequality
 \begin{equation}\label{BR-equ-2.2}
 \int_X|\nabla f|^2\Delta g{\rm d}\meas\gs \int_X\Big( \frac{(\Delta f)^2}{N} +\ip{\nabla f}{\nabla \Delta f}+K|\nabla f|^2\Big)g{\rm d}\meas
 \end{equation}
     for any $g\in D(\Delta )\cap L^\infty(X)$ with $\Delta g\in L^\infty(X)$ and $g\gs 0$.
 \item   There is a point $x_0\in X$ such that $\meas(B_r(x_0))\ls c_1 e^{c_2r^2}$ for all $r>0$, for some   constants $c_1,c_2>0$.
 \item For any $f\in W^{1,2}(X)$ with $|\nabla f|\ls 1$ $\meas-$a.e. in $X$, there exists a representative $g\in Lip(X)$ with Lipschitz constant 1 such that $g=f$ $\meas-$a.e. in $X$.
  \end{enumerate}
\end{defn}

The main examples in the class of $RCD(K,N)$ spaces include the Ricci limit spaces of the Cheeger-Colding theory \cite{Stu06b,LV09,AGS14a,AGS14b} and finite-dimensional Alexandrov spaces with curvature bounded from below \cite{Pet11,ZZ10}.

From now on, we assume always that $(X,d,\meas)$ is an $RCD(K,N)$ space for some $K\in\mathbb R$ and $N\in [1,+\infty)$.
When $N>1$, it holds the generalized Bishop-Gromov inequality (see, for example, \cite{EKS15}) that
\begin{equation}\label{BR-equ-2.3}
\frac{\meas\big(B_R(x)\big)}{V_{N,K}(R)}\ls \frac{\meas\big(B_r(x)\big)}{V_{N,K}(r)}, \qquad \forall 0<r<R,\quad \forall x\in X,
\end{equation}
where the function
$V_{N,K}(R)$ is given by
\begin{equation*}
V_{N,K}(R):=\int_0^R\mathfrak{s}^{N-1}_{\frac{K}{N-1}}(\tau){\rm d}\tau,\qquad
\mathfrak{s}_{k}(\tau)=
\begin{cases}
\frac{\sin(\sqrt k\cdot \tau)}{\sqrt k}& {\rm if} \quad k>0,\\
\tau & {\rm if} \quad k=0,\\
\frac{\sinh(\sqrt {-k}\cdot \tau)}{\sqrt {-k}}& {\rm if} \quad k<0.
\end{cases}
\end{equation*}
In particular, if $K\ls0$ and $N>1$, it implies the doubling property of $\meas$:  for any $r<s<R$,
 \begin{equation}\label{BR-equ-2.4}
 	\frac{\meas(B_{s}(x))}{\meas(B_{r}(x))}\ls \exp\left(\sqrt{-K(N-1)}\cdot R\right) \left(\frac{s}{ r}\right)^N:=C_{N,K,R}  \cdot\left(\frac{s}{ r}\right)^N.
 	 \end{equation}
It was proved in \cite{Raj12} that any $RCD(K,N)$ space with $K\ls0$ supports the following locally weak (1,1)-Poincar\'e inequality:
\begin{equation}\label{equ-add-Poincare-ineq}
	\fint_{B_r(x)}|f-f_{B}|{\rm d}\meas\ls 2^{N+2} r e^{2\sqrt{(N-1)|K|}r}\fint_{B_{2r}(x)}|\nabla f|{\rm d}\meas
\end{equation}
for any locally Lipschitz function $f$ on $X$, where $f_B:=\fint_{B_r(x)}f{\rm d}\meas.$

 \subsection{Calculus on a local domain $\Omega\subset X$}

\begin{defn}[Local Sobolev Space] \label{BR-def-2.2}
 Let $\Omega\subset X$ be an open set. A  function $f\in L^2_{\rm loc}( \Omega)$ belongs to
$W^{1,2}_{\rm loc}(\Omega)$, provided, for any    $\chi\in Lip_0(\Omega)$   it holds
$f\chi\in W^{1,2}(X)$, where $f\chi$ is understood to be  $0$ outside of $\Omega$.
\end{defn}

Let $f\in W^{1,2}_{\rm loc}(\Omega)$,  the function $|\nabla f| : \Omega\to[0,\infty]$ is $\meas$-a.e. defined by
$$|\nabla f| := |\nabla (\chi f)|,\ \ \meas-a.e. \,\mathrm{on} \ \{\chi=1\},$$
 for any $\chi$ as above.
The space
$$W^{1,2}(\Omega):=\{f\in W^{1,2}_{\rm loc}(\Omega)\big|\  f, |\nabla f|\in L^2(\Omega)\},$$
and the space $W_0^{1,2}(\Omega)$ is defined as the $W^{1,2}(X)$-closure of the space of functions $f\in Lip_0(\Omega)$, where $f\in Lip_0(\Omega)$ is understood to be $0$ outside of $\Omega$.

Recall from \cite{KM96} that the \emph{Sobolev 2-capacity} of the set  $E\subset X$:
$${\rm Cap}_2(E):=\inf\big\{\|f\|^2_{W^{1,2}(X)}\big|\ f\in W^{1,2}(X)\ {\rm such \ that }\ f\gs 1\ {\rm on\ a\ neighborhood\ of }\ E\big\}.$$
If there is no such a function $f$, we set ${\rm Cap}_2(E)=\infty.$

A property holds $2$-q.e. ($2$-quasi everywhere), if it holds except of a set $Z$ with ${\rm Cap}_2(Z)=0$. Since ${\rm Cap}_2(\overline{Z})={\rm Cap}_2(Z),$ we may assume that the except set $Z$ is closed. A function $f :X\to[-\infty,\infty]$ is called $2$-$quasi\ continuous$ in $X$ if for each $\epsilon>0$, there is a closed set $F_\epsilon$ such that ${\rm Cap}_2(F_\epsilon)<\epsilon$ and the restriction $f|_{X\backslash F_\epsilon}$ is continuous.
It is well-known (see \cite{KM96}) that for any $W^{1,2}(X)$-function $f$ its Lebesgue representation
$$\lim_{r\to0}\fint_{B_r(x)}f(y)\meas(y)$$
exists $2$-q.e. in $X$ and is
 $2$-quasi continuous. We will always use such a representative in this paper. Using this a representation, the $W^{1,2}_0(\Omega)$-function have the following characterizations (see \cite{KKM00, KKST12}).

 \begin{lemma}\label{BR-lem-2.3}

 Let $\Omega$ be a bounded subset in $X$. Then the following are equivalent:
 \begin{itemize}
 	\item [(i)] $f\in W^{1,2}_0(\Omega)$;
 	\item [(ii)]
 there exists  $2$-quasi continuous function $\tilde f\in W^{1,2}(X)$ such that $\tilde f=f$ $\meas$-a.e. in $\Omega$ and $\tilde f$=0 $2$-q.e. in $X\setminus\Omega;$
 \item[(iii)] 	  its zero extension $\bar{f}$ is in  $W^{1,2}(X)$. Here the zero extension $\bar f$ is given by   $\bar f=f$ in $\Omega$ and $\bar f=0$ in $X\setminus\Omega$.
   \end{itemize}
  \end{lemma}
  \begin{proof}
  $({\rm i})\Leftrightarrow ({\rm ii})$ is proved in   \cite[Remark 5.10]{KKM00}, since $W^{1,2}(X)$ is releflexive.

  $({\rm iii})\Rightarrow ({\rm ii})$ is obvious, by using the $2$-quasi contnuious representation of $\bar f$.

  For $({\rm ii})\Rightarrow ({\rm iii})$, let $f\in W^{1,2}_0(\Omega)$ and let $\tilde f$ is given in (ii). Therefore $\tilde f= \bar f$ $\meas$-a.e. in $X$, which implies $\bar f\in W^{1,2}(X)$.
     \end{proof}

We fix a bounded domain $\Omega\subset X$ satisfying  ${\rm diam}(\Omega)\ls {\rm diam}(X)/a$ for some $a>1$ and $\meas(\partial\Omega)=0$. The canonical Dirichlet form, $(\mathscr E_{\Omega},W^{1,2}_0(\Omega))$, is given by
\begin{equation*}
\mathscr E_\Omega(f):=\int_\Omega|\nabla f|^2{\rm d}\meas,\qquad f\in W^{1,2}_0(\Omega).
\end{equation*}
This canonical Dirichlet form is strongly local and regular   (see, for example, the proof of \cite[Lemma 6.7]{AGS14b}).
  The associated infinitesimal generator of $(\mathscr E_{\Omega}, W^{1,2}_0(\Omega))$, denoted by $\Delta_\Omega$ with domain $D(\Delta_\Omega)$, is a non-positive definite self-adjoint operator. The associated analytic semi-group is $(H^\Omega_tf)_{t\gs0}$ for any $f\in L^2(\Omega)$.
Since $\Omega$ is bounded, from   \cite[Theorem 13.1]{HK00}, the embedding $W^{1,2}_0(\Omega)\subset L^2(\Omega)$ is compact.   Hence the operator $(Id-\Delta_\Omega)^{-1}$ is compact. The spectral theorem implies that the spectrum is discrete. We denote by
$$0<\lambda_1^{\Omega}\ls \lambda_2^{\Omega}\ls \cdots\ls \lambda_j^{\Omega}\ls \cdots, \quad j\in \mathbb N,$$
the (Dirichlet) eigenvalues of  $\Delta_\Omega$. For each $\lambda_j^{\Omega}$, the associated eigenfunction is $\phi_j^{\Omega}$, i.e.,
\begin{equation}\label{BR-equ2.5}
\Delta_\Omega \phi_j^{\Omega}=- \lambda_j^{\Omega}\phi_j^{\Omega}.
\end{equation}
 We normalize them so that $\|\phi_j^{\Omega}\|_2=1 $ for each $j\in\mathbb N$. It is well-known that the sequence $\{\phi_j\}_{j\in\mathbb N}$ forms a complete basis of $L^2(\Omega)$, and that the (local) Dirichlet heat flow is given by
 $$ H^\Omega_tf(x)=\int_\Omega p^\Omega_t(x,y)f(y){\rm d}\meas,\quad t\geqslant 0,\ \ \forall f\in L^2(\Omega),$$
  where
\begin{equation}\label{BR-equ-2.6}
p^{\Omega}_t(x,y)=\sum_{j\gs1}e^{-\lambda^{\Omega}_j t}\phi^{\Omega}_j(x)\phi^{\Omega}_j(y),\qquad \forall (x,y,t)\in \Omega\times\Omega\times(0,\infty).
\end{equation}
 is the (local) Dirichlet heat kernel. It satisfies the semi-group property
 $$p^\Omega_{t+t_0}(x,y)=H^\Omega_tp^\Omega_{t_0}(x,\cdot)(y) $$
 and  the $L^2$-contraction
 \begin{equation}\label{BR-equ-2.7}
 \|H^\Omega_tf\|_{L^2(\Omega)}\ls e^{-\lambda_1^\Omega\cdot t}\|f\|_{L^2(\Omega)},\quad \forall f\in L^2(\Omega).
 \end{equation}

 The weak maximum principle and the Gaussian upper bound of heat kernel in \cite{Stu95, JLZ16} imply that
\begin{equation}\label{BR-equ-2.8}
p^{\Omega}_t(x,y) \ls  \frac{C_0\cdot e^{C_0t}}{\meas\left(B_{\sqrt t}(x)\right)}\exp\left(-\frac{d^2(x,y)}{5t}\right)	
\end{equation}
 for any  $(x,y,t)\in \Omega\times\Omega\times(0,\infty)$, where $C_0$ depends only on $N,K$.

	The class of {\textit{test  functions}} on $RCD(K,N)$ spaces was introduced in \cite{AGS15,Sav14}:
	\begin{equation*}
		\mathrm{Test}^\infty(X):=\left\{f\in D(\Delta)\cap L^\infty(X)\big|\ |\nabla f|\in L^\infty(X),\ \Delta f\in L^\infty(X)\cap W^{1,2}(X)\right\},
	\end{equation*}
	which is an algebra and is dense in $W^{1,2}(X)$ (see \cite[Sect. 6.1.3]{GP20}).  Given any  $\phi\in \mathrm{Test}^\infty(X)$, it has a (Lipschitz) continuous representative. Hence we always assume that every function in  $  \mathrm{Test}^\infty(X)$ is continuous. For any given open subset $\Omega\subset X$, we set
	\begin{equation*}
			\mathrm{Test}_c^{\infty}(\Omega):=\left\{\phi \in  \mathrm{Test}^{\infty}(X) :\  {\rm    supp}(\phi)\subset \Omega \ {\rm and\ supp}(\phi) \ {\rm is\ compact}\right\}.
	\end{equation*}
It is clear that if $\phi\in {\rm Test}_c^\infty(\Omega)$ then it is in $D(\Delta_\Omega)$ and $\Delta_\Omega\phi=(\Delta\phi)|_{\Omega}.$

 Let us recall the definition of the  \emph{distributional Laplacian}.
 \begin{defn}\label{BR-def-2.4}
 Given a function $f\in L^{1}_{\rm loc}(\Omega)$,  its distributional Laplacian $\mathbf\Delta  f$ is defined as a (linear) functional
\begin{equation*}
{\bf \Delta} f(\phi):=\int_\Omega f\Delta \phi {\rm d}\meas,\qquad \forall \phi\in {\rm Test}_c^\infty(\Omega).
\end{equation*}
 \end{defn}
\noindent When $f\in W_{\rm loc}^{1,2}(\Omega)$,  $\mathbf\Delta f$ can be extended to a functional on $Lip_0(\Omega).$
 Therefore, the definition of ${\bf \Delta}f$ is compatible with the one in \cite{Gig15}.	When  $f\in W^{1,2}(\Omega)$, $\mathbf\Delta f$ can be extended to a functional on $W^{1,2}_0(\Omega).$

Given $f\in L^{1}_{\rm loc}(\Omega)$ and a signed Radon measure $\mu$, the notion ``${\bf \Delta} f \gs \mu$ in the sense of distributions" means that
 $$\int_{\Omega}  f \Delta \phi {\rm d}\meas\gs \int_\Omega  \phi {\rm d}\mu,\qquad \forall\ \phi\in {\rm Test}^\infty_c(\Omega),\ \phi\gs0.$$	
In this case, the functional ${\bf \Delta}f$ provides a signed Radon measure on $\Omega$, denoted by ${\bf \Delta}f$ again  (by Riesz representation theorem). Therefore, in this case, we write also  ``${\bf \Delta}f\gs \mu$ as measures". If the measure $\mu=g\cdot\meas$ for some $g\in L^1_{\rm loc}(\Omega)$, we  denoted by ``${\bf \Delta} f \gs g$ in the sense of distributions" too.

Recall the Maximum principle as follows (see \cite{Che99}).   Let $f\in W^{1,2}(\Omega)$ and  ${\bf \Delta}f\gs 0$ as measures. If $f\ls c$ on $\partial \Omega$ in the sense of $(f-c)^+\in W^{1,2}_0(\Omega)$  then $f\ls c$, $\meas$-a.e. in $\Omega.$ In particular, if $f\in W^{1,2}(\Omega)\cap C(\overline{\Omega})$,   if $f\ls c$ on $\partial \Omega$, and if ${\bf \Delta} f\gs0$ as measures, then $f\ls c$ on $\Omega.$

When $f\in W^{1,2}_0(\Omega)$, it was proved \cite{Gig15} that if there is   $g\in L^2(\Omega)$ such that $\mathbf \Delta f= g$ in the sense of distributions  then $f\in D(\Delta_{\Omega})$ and $\Delta_\Omega f=g.$
 Conversely, it is clear that if $f\in D(\Delta_{\Omega})$ and $\Delta_\Omega f=g,$ then $\mathbf \Delta f= g$ in the sense of distributions.

\begin{lemma}
\label{BR-lem-2.5}
Let $\Omega$ be a bounded domain of $(X,d,\meas)$.  Suppose that  $f\in L^{1}_{\rm loc}(\Omega)$ satisfies  $\mathbf \Delta f=h$ in the sense of distributions for some $h\in L^\infty_{\rm loc}(\Omega)$. Then it holds:
 \begin{itemize}
 	\item [(1)]    $f $ is in $Lip_{\rm loc}(\Omega)$ (see \cite[Corollary 1.5]{PZZ25}).
 	\item [(2)]  Suppose that $ f$ is harmonic on $\Omega$ (i.e. ${\bf \Delta} f=0$) and $f-g\in W^{1,2}_0(\Omega)$ for some $g\in W^{1,2}(\Omega)\cap C(\overline\Omega)$, if $\Omega$ satisfies a {\emph{unform exterior density condition}}, that is,
 	$$\meas\big(B_r(x)\setminus\Omega\big)\gs C\meas\big(B_r(x)\big),\qquad \forall r\in(0,r_0),\ \forall x\in\partial\Omega$$ for some
 	$r_0>0$ and $C>0$,
 	then $f\in C(\overline\Omega)$ (see \cite[Theorem 2.13 and Remark 2.15]{Bjo02}).
\end{itemize}
 \end{lemma}

The following corollary of the Laplacian comparison theorem \cite{Gig15} will be used.

\begin{lemma}\label{BR-lem-2.6}
For each $y_0\in X$, $R\in(0,1)$ and $R_1>R$, there exists a function $h\in Lip_0(X)$ satisfying the following properties:
\begin{itemize}
	\item [(i)] $h\gs 0$ on $X$, $h=0$ on $B_{R/2}(y_0)$, and  $h\gs 1 $ on $B_{R_1}(y_0)\setminus B_R(y_0)$,
	\item [(ii)] $|\nabla h|\ls \frac{C_{1}}{R}$ $\meas$-a.e. in $B_R(y_0)\setminus \overline{B_{R/2}(y_0)}$,
	\item [(iii)] ${\bf\Delta} h\ls -\frac{2}{R^2}\cdot \meas$ in $B_R(y_0)\setminus \overline{B_{R/2}(y_0)}$ in the sense of distributions,
\end{itemize}
where the constant $C_1$ depends only on $N$ and $K$.
\end{lemma}
\begin{proof}
From the Laplacian comparison \cite{Gig15} and $R\ls 1$, we have
$${\bf \Delta}\rho \ls  C\rho^{-1}\cdot \meas \quad {\rm on}\ \ B_R(y_0)\setminus\{y_0\},$$
in the sense of distributions, where $\rho$ is the distance function $\rho(x):=d(y_0,x)$ and the constant $C>0$ depends only on $N, K$. By the Chain rule \cite{Gig15} and letting $l:=C+1\gs 1$, we have
\begin{equation*}
	\begin{split}
		{\bf \Delta}\rho^{-l}&=(-l)\rho^{-l-1}\cdot {\bf \Delta}\rho+(-l)(-l-1)\cdot\rho^{-l-2}|\nabla \rho|^2\cdot\meas\\
		&\gs l(-C+l+1)\rho^{-l-2}\cdot \meas =2l\cdot \rho^{-l-2}\cdot\meas	\end{split}
\end{equation*}
in the sense of distributions on $B_R(y_0)\setminus\{y_0\}$. We put
$h_1(x):=2^l-R^l\cdot \rho^{-l}(x)$
and then
$${\bf \Delta}h_1\ls - \frac{2l}{R^2}\left(\frac R \rho\right)^{l+2}\cdot\meas \ls -\frac{2}{R^2}\cdot \meas,$$
in the sense of distributions on $B_R(y_0)\setminus\{y_0\}$, where we have used $\rho\ls R$ and $l\gs1$. At last, we define $h:=\eta \cdot(h_1)_+$, where  $\eta(x) $ is a nonnegative function in ${\rm Test}^\infty_c(X)$ such that $\eta(x)\equiv 1$ in $B_{R_1}(y_0)$. It is easy to check that $h$ satisfies all assertions with $C_1:= l\cdot 2^{l+1}.$
\end{proof}

 \subsection{Sets of finite perimeters and BV functions}

The theory of sets of finite perimeter and functions of bounded variation has been generalized to the setting of  $RCD$ spaces (\cite{ABS19, MMS16, BCM22, BPS23-jems, GM24}).
\begin{defn}\label{BR-def-2.7-BV}
A function $f\in L^1(X)$ is called  a function {\emph{of bounded variation}}, denoted by $f\in BV(X)$, if there exist $f_j\in Lip_{\rm loc}(X)$ converging to $f$ in $L^1(X)$ such that
$$\limsup_{j}\int_X|\nabla f_j|{\rm d}\meas<+\infty.$$
\end{defn}

Given a function $ f \in BV(X)$  and an open
set $A\subset X$, one can define
$$|Df|(A):= \inf\left\{ \liminf_{i\to+\infty}\int_A|\nabla f_i|{\rm d}\meas\ \big|\ f_i\in Lip_{\rm loc}(A).\quad f_i\to f\ \ {\rm in}\ \ L_{\rm loc}^1(A)\right\}.
$$

Indeed, (see \cite{Mir03, ABS19}), the set function $A\mapsto  |Df|(A)$  is the restriction to open sets of a finite Borel measure, which is called the total variation of $f$ and denoted still by $|Df|$. For any Borel set $B$,
$$|Df|(B):= \inf \left\{|Df|(A) \ \big|\  B \subset A,\  A\ {\rm open}\right\}.$$

\begin{defn}\label{BR-def-2.8}
	A Borel set $E\subset X$ with $\meas(E)<\infty$  is called a {\emph {set of finite perimeter}} if $\chi_E\in BV(X)$. In this case, the $|D\chi_E|$ is called the perimeter measure, denoted by $Per_E$.
\end{defn}

 Let $E$ be a set of finite perimeter and $0\ls t\ls 1$. We denote by
 \begin{equation}\label{BR-equ-2.9}
 E^{(t)}:=\left\{ x\in X \big | \ \lim_{r\to0^+}\frac{\meas(E\cap B_r(x))}{\meas(B_r(x))}=t\right\}.
 \end{equation}
The  {\emph{measure-theoretic boundary}}  $\partial^*E:=X\setminus (E^{(0)} \cup E^{(1)})$. Denote
the {\emph{reduced boundary}} of $E$  by $\mathscr FE$, which is defined in  \cite{BPS23-jems,BPS23} to be a subset of $E^{(1/2)}$.   Theorem 3.2 in \cite{BPS23-jems} asserts $Per_E(X\setminus \mathscr FE)=0.$

We also recall the codimension-$1$ Hausdorff measure of $\meas$.
 \begin{defn}\label{BR-def-2.9}
 Given a set $E\subset X$ and  $\alpha\gs0$, the codimension-$\alpha$ Hausdorff measure with respect to $\meas$ is defined by
 	 	 $$\mathscr \meas_{-\alpha}(E):=\lim_{\delta\to0}\mathscr \meas_{\delta,-\alpha}(E),$$
 	 	 where
 	 	 $$\meas_{\delta,-\alpha}(E):=\inf\left\{\sum_{i=1}^\infty\frac{\meas(B_{r_i}(x_i))}{r_i^\alpha}\ \big| \ E\subset\cup_{i=1}^\infty B_{r_i}(x_i),\ 0<r_i\leqslant \delta\right\}.$$
 	 \end{defn}

 Let $E\subset X$ be a set of finite perimeter with $\meas(E)<+\infty$.
 It was proved in \cite{BPS23, BPS23-jems} that
 $$\meas_{-1}\big(X\setminus(E^{(1)}\cup E^{(1/2)}\cup E^{(0)})\big) =0,$$
 and
  \begin{equation*}
 	Per_E(A)=c_n\cdot \meas_{-1}(\partial^*E\cap A),\quad \forall A\ \ {\rm open},
 \end{equation*}
 for  some constant $c_n>0$ depending only on the essential dimension of $X$ (see \cite{BS20} for the definition of the essential dimension).

\section{Boundary estimates for Dirichlet heat kernels and Green functions}

 In this section, we will extend the research of the Dirichlet heat kernels and the Green functions to the $ RCD$ setting.
Let $(X,d,\meas)$ be an $RCD(K, N)$ space. Without loss of generality, we assume that $K<0$ and $N>1$. Let $\Omega$ be a bounded domain satisfying $\meas(\partial\Omega)=0$ and ${\rm diam}(\Omega)\ls {\rm diam}(X)/a$ for some $a>1$.
We also assume that $\Omega$ satisfies the uniform exterior ball condition with radius $R_{\rm ext}\in(0,1)$ in Definition \ref{BR-def-1.5}.

\subsection{Boundary behaviors of Dirichlet heat kernels}
Let us begin with the parabolic comparison principle on metric measure spaces, which is one of useful tools in partial differential equations.

\begin{defn}
	\label{BR-def-3.1}
Let $0<T<\infty$ and let $\Omega_T:=\Omega\times(0,T)$. A function $u\in L^2_{\rm loc}\big((0,T); W^{1,2}_{\rm loc}(\Omega)\big)$ is called a {\emph{super-solution of the heat equation}}, if for every open set $U\Subset\Omega_T$ and for all non-negative function $\phi(x,t)\in Lip(\Omega_T)$ with ${\rm supp}(\phi)\subset U$, we have
$$\int_U u\frac{\partial\phi}{\partial t}{\rm d}\meas {\rm d}t\ls \int_U\ip{\nabla u}{\nabla \phi}
	{\rm d}\meas {\rm d}t.$$
A function $u$ is called a {\emph{sub-solution of the heat equation}} if $-u$ is a super-solution of the heat equation.
\end{defn}

 We recall the parabolic comparison principle given in \cite{KM15-pams}.
\begin{lemma}\label{BR-lem-3.2}
	Let $u\in L^2\big((0,T); W^{1,2}(\Omega)\big)$ be a super-solution of heat equation and let $v\in L^2\big((0,T); W^{1,2}(\Omega)\big)$ be a sub-solution of heat equation. Suppose $u\gs v$ near the parabolic boundary of $\Omega_T$ in the sense that for almost every $0<t<T$ we have the lateral boundary condition
	$$\big(v(\cdot,t)-u(\cdot,t)\big)_+\in W^{1,2}_0(\Omega)$$ and also
	the initial condition
	$$\frac 1 \varepsilon\int_0^\varepsilon\int_\Omega(v-u)_+^2{\rm d}\meas {\rm d}t\to 0,\quad {\rm as}\ \ \varepsilon\to0^+.$$
		Then
		$$u\gs v,\quad \meas\times \mathscr L^1{\rm-a.e.\  in}\ \  \Omega_T,$$ where $\mathscr L^1$ is the 1-dimensional Lebesgue measure on $(0,T)$.
		\end{lemma}

\begin{proof}
	This is \cite[Theorem 4.1]{KM15-pams}.
\end{proof}

To use the parabolic comparison principle to estimate a solution to the heat equation, the key is to find an appropriate barrier function. The following lemma gives some barriers.

\begin{lemma}\label{BR-lem-3.3}
Let $y_0\in X$, $R\in(0,1)$, $R_1>R$ and let $h$ be the function given in Lemma \ref{BR-lem-2.6}. Let $T>2R^2$. Then
$$\tilde h(x,t):=h(x)+\left(\frac{T-t}{R^2}\right)^2$$
is a super-solution of the heat equation on $A_R(y_0)\times (T-R^2, T)$, where $A_R(y_0):=B_{R}(y_0)\setminus\overline B_{R/2}(y_0)$.
	\end{lemma}
\begin{proof}
	By Lemma \ref{BR-lem-2.6}, we have that $\tilde h\in W^{1,2}\big(A_R(y_0)\times (0,T)\big)$ and that for each $t>0$,
	$${\bf \Delta}\tilde h(x,t)= {\bf \Delta}h\ls -\frac{2}{R^2}\cdot \meas
	$$ on $A_R(y_0)$ in the sense of distributions. For each $t\in(T-R^2,T)$,
	$$\frac{\partial}{\partial t}\tilde h(x,t)= \frac{2(t-T)}{R^4}\gs \frac{2(T-R^2-T)}{R^4}=-\frac{2}{R^2}.$$
	It follows  for each $t\in(T-R^2,T)$  that  ${\bf \Delta}\tilde h\ls\frac{\partial}{\partial t}\tilde h\cdot \meas$ in the sense of distributions. Therefore, it is a super-solution of the heat equation (see, for example,  \cite[Lemma 6.12]{ZZ18} or \cite[Lemma 6.2]{MS22+}).
	 	\end{proof}

By the comparison principle and the barriers above, we can get the following $L^\infty$ self-improvement property for the solutions to the heat equation near the boundary.

\begin{theorem}
	\label{BR-thm-3.4}
	Let $v(x,t)\in  L^2\big((0,T_0); W^{1,2}(\Omega)\big)$ be a sub-solution of heat equation. Suppose that $v\in C(\Omega\times(0,T_0))$ and that $v(\cdot,t)\in W^{1,2}_0(\Omega)$ for almost every $0<t<T_0$. Let $x_0\in \partial \Omega$ and suppose that $\Omega$ satisfies the exterior ball condition at $x_0$ with radius $R_{\rm ext}$.  Let $0<R<\min\{R_{\rm ext},{\rm diam}(\Omega)\}$ and $2R^2\ls T<T_0$.
		Assume that
	$$v(x,t)\ls a\quad {\rm on}\quad (B_{2R}(x_0)\cap \Omega)\times (T-R^2,T)$$ for some $a>0$. Then for any $r\in(0,R/2)$ it holds
	$$v(x, T)\ls \frac{c_1\cdot a r}{R}  \quad {\rm on}\quad B_r(x_0)\cap \Omega,$$
	where $c_1$ depends only on $K$ and $N$.
	\end{theorem}

\begin{proof}
The proof is similar to the argument in \cite[Section 1]{Hui92}. In \cite{Hui92}, the barrier function is given via the boundary Schauder's estimate, which does not work for the $RCD$ setting. Here, we will use the barrier function $\tilde h$ given in the above Lemma \ref{BR-lem-3.3}.

	Let $y_0\in X\setminus\Omega$ be a point such that  $d(y_0,x_0)=d(y_0,\Omega)= R/2$. Denote by $U:=A_R(y_0)\cap \Omega.$ Then $U=B_R(y_0)\cap\Omega,$ since $\Omega\cap \overline{B_{R/2}(y_0)}=\emptyset.$ Remark that $U\subset B_{2R}(x_0)\cap \Omega.$ Let $\tilde h$ be the function given in Lemma \ref{BR-lem-3.3} with $R_1=2\cdot{\rm diam}(\Omega)$.
	
	We first claim that $v\ls a\tilde h$ near the parabolic boundary of $U\times (T-R^2,T)$ in the sense of Lemma \ref{BR-lem-3.2}.
	
	 Let $t\in (T-R^2,T)$ such that $v(\cdot,t)\in W^{1,2}(\Omega)$. We can consider its zero extension $\bar v$. Then, by Lemma \ref{BR-lem-2.3}(iii), $\bar v\in W^{1,2}(X)$. We define
	$$w:=\big(\min\{\bar v,a\}- a\tilde h\big)_+.$$
	It is clear that $w\in W^{1,2}_{\rm loc}(X)$. Noticing that $w=0 $ in $X\setminus\Omega$, (since $\bar v=0$ in $X\setminus\Omega$ and $\tilde h\gs0$), we conclude that $w\in W^{1.2}(X)$. 	Since $\tilde h(x,t)\gs h(x)\gs 1$ on $B_{R_1}(y_0)\setminus  B_R(y_0)\supset \Omega\setminus B_R(y_0)$, we get that $w=0$ on  $\Omega\setminus B_R(y_0)$.
	 Therefore, $$w=0 \quad {\rm on}\quad   (X\setminus \Omega)\cup(\Omega\setminus B_R(y_0))=X\setminus U.$$
	   The assumption $v\ls a$ on $U\times (T-R^2,T)$ implies that  $\min\{\bar v(\cdot, t),a\}=v(\cdot,t)$ on $U$ for each $t\in (T-R^2,T),$ and hence  $w=(v-a\tilde h)_+$ on $U$ for each $t\in(T-R^2,T)$. By Lemma \ref{BR-lem-2.3}(i), we conclude that lateral boundary condition $(v-a\tilde h)_+\in W^{1,2}_0(U).$
	
Considering the initial condition,  for any small $\varepsilon>0$, we have
$$\tilde h(x,t)\gs h(x)+(\frac{-R^2+\varepsilon}{R^2})^2\gs (1-\varepsilon/R^2)^2$$
 on $U\times(T-R^2,T-R^2+\varepsilon)$. By the assumption $v\ls a $ on $U\times (T-R^2,T)$, we get
	$$(v-a\tilde h)_+\ls a \left(1-(1-\varepsilon/R^2)^2\right)$$
	on $ U\times(T-R^2,T-R^2+\varepsilon)$, which implies the initial condition
	$$\frac 1 \varepsilon\int_{T-R^2}^{T-R^2+\varepsilon}\int_U(v-a\tilde h)_+^2{\rm d}\meas{\rm d}t\to 0\quad {\rm as}\  \ \varepsilon\to0.$$  Now we finish the proof of the claim that $v\ls a\tilde h$ near the parabolic boundary of $U\times (T-R^2,T)$ in the sense of Lemma \ref{BR-lem-3.2}.
	
	By Lemma \ref{BR-lem-3.2}, we conclude that $v\ls a\tilde h$ almost all in $U\times (T-R^2,T)$. Since both $v$ and $\tilde h$ are continuous on $U\times(0,T_0)$, we have
	$$v(x,T)\ls a\tilde h(x, T) =ah(x)\quad {\rm on} \ \ U.$$
Recalling $|\nabla h|\ls c_1/R$ on $U\subset A_R(y_0)$, $h(x_0)=0$, and noticing   $B_r(x_0)\cap \Omega\subset U$ for $r<R/2,$ this finishes the proof.
\end{proof}

Now we give the upper bounds of the Dirichlet heat kernel near the boundary, which is the main result in this subsection.

\begin{theorem}\label{BR-thm-3.5}	
Let $\Omega$ be bounded and satisfy a uniform exterior ball condition with radius $R_{\rm ext}\in (0,1)$. Then $p^\Omega_t(x,y)$ satisfies the following upper bounds: for any $x,y\in \Omega$ and any $\widehat T\gs R^2_{\rm ext}$ it holds
\begin{equation}
	\label{BR-equ-3.1}
	p^\Omega_T(x,y)\ls c\frac{\delta(x)}{\sqrt{T}}\cdot\frac{ e^{c_2  T}\cdot {\sqrt{\widehat T}}/R_{\rm ext} }{\meas(B_{\sqrt T}(y))} \exp\left(-\frac{d^2(x,y)}{c_1 T}\cdot \frac{\widehat T}{R^2_{\rm ext}} \right)
\end{equation}
and
\begin{equation}
	\label{BR-equ-3.2}
	p^\Omega_T(x,y)\ls c\frac{\delta(x)\delta(y)}{T}\cdot\frac{  e^{c_2  T}\cdot {\widehat T}/R^2_{\rm ext}}{\meas(B_{\sqrt T}(y))} \exp\left(-\frac{d^2(x,y)}{c_1 T}\cdot \frac{\widehat T}{R^2_{\rm ext}} \right)
\end{equation}
for all $T\in(0,\widehat T)$, where $\delta(x):=d(x,\partial \Omega),$ and the constants $c,c_1, c_2$ depend only on $N$ and $K$.

\end{theorem}

\begin{proof}
We first know that $p^\Omega_t(x,y)\in C(\Omega\times(0,+\infty))$ and for each $t>0$ that $p^\Omega_t(x,\cdot)\in W^{1,2}_0(\Omega).$

Fix any $y\in\Omega$ and  $T<\widehat T$. Let $a= \frac{\sqrt{\widehat T}}{R_{\rm ext}}$ and $R:= \frac{\sqrt T }{2a}\ls \frac{R_{\rm ext}}{2}.$
 We can assume that $\delta(x)<R/4$, otherwise we have done by (\ref{BR-equ-2.8}).   Let $x_0 \in\partial\Omega$ be the point such that $d(x,x_0)=\delta(x).$

  From (\ref{BR-equ-2.8}),
$$p^\Omega_t(z,y)\ls \frac{c }{\meas(B_{\sqrt{t}}(y))}\exp\left(-\frac{d^2(z,y)}{5t}+ct\right), \quad \forall z\in \Omega.$$
Since $T\gs 2a^2R^2\gs 2R^2$, it holds $T/2\ls t\ls T$ for any $t\in (T-R^2,T)$. Hence, by the doubling property of $\meas$ (\ref{BR-equ-2.4}), we  have
$$p^\Omega_t(z,y)\ls \frac{c\cdot e^{cT}}{\meas(B_{\sqrt{T}}(y))}\exp\left(-\frac{d^2(z,y)}{cR^2}\right), \qquad \forall (z,t)\in\Omega\times(T-R^2,T).$$
Now we consider two cases: $d(x,y)\gs 8R$ or $d(x,y)<8R$.
In the case when $d(x,y)\gs 8R$, it holds for any $z\in B_{2R}(x_0)$ that
$$d(z,y)\gs d(x,y)-d(x,z)\gs d(x,y)-3R\gs \frac{d(x,y)}{2},$$
since $d(x,x_0)=\delta(x)<R$. Therefore, we have
$$p^\Omega_t(z,y)\ls \frac{c\cdot e^{cT}}{\meas(B_{\sqrt{T}}(y))}\exp\left(-\frac{d^2(x,y)}{cR^2}\right), \qquad \forall (z,t)\in B_{2R}(x_0)\times(T-R^2,T).$$
In the case when $d(x,y)\ls 8R$, we have $\exp(-\frac{d^2(x,y)}{R^2})\gs \exp(-1/8^2)$, and hence
$$p^\Omega_t(z,y)\ls \frac{c\cdot e^{cT}}{\meas(B_{\sqrt{T}}(y))}\ls\frac{c_1\cdot e^{cT}}{\meas(B_{\sqrt{T}}(y))}\exp\left(-\frac{d^2(x,y)}{cR^2}\right)$$
for any $ (z,t)\in \Omega\times(T-R^2,T).$ Therefore, in both cases, we have
$$p^\Omega_t(z,y)\ls \frac{c\cdot e^{cT}}{\meas(B_{\sqrt{T}}(y))}\exp\left(-\frac{d^2(x,y)}{cR^2}\right), \qquad \forall (z,t)\in B_{2R}(x_0)\times(T-R^2,T).$$

By applying Theorem \ref{BR-thm-3.4} to  $v(\cdot, t):=p^\Omega_t(\cdot, y)$ on $B_{2R}(x_0)\times (T-R^2, T)$ with $r=\delta(x)< R/2$, we conclude that
$$p^\Omega_T(x,y)\ls c_1\frac{\delta(x)}{R}\frac{c\cdot e^{cT}}{\meas(B_{\sqrt{T}}(y))}\exp\left(-\frac{d^2(x,y)}{cR^2}\right).$$
This implies (\ref{BR-equ-3.1}), since $T=4a^2R^2$ and $a^2=\frac{\widehat T}{R^2_{\rm ext}}$.

The proof of \eqref{BR-equ-3.2} is similar to the proof of (\ref{BR-equ-3.1}), by considering the function $p^\Omega_t(x,\cdot)$ near $y$ and replacing (\ref{BR-equ-2.8})  by (\ref{BR-equ-3.1}).
\end{proof}

\begin{proof}[Proof of Theorem \ref{BR-thm-1.10}]
	By replacing $T=t$ and $\widehat  T=T$ in Theorem \ref{BR-thm-3.5}, one can get  Theorem \ref{BR-thm-1.10}.
 \end{proof}

\subsection{Boundary behaviors of Dirichlet Green functions}

The Dirichlet Green functions on metric measure spaces have been studied widely (see, for example, \cite{BM95, G-H14, CGL21}).

From the positivity of heat semi-group $\{H^\Omega_t\}$,  the local Dirichlet heat kernel $p^\Omega_t(x,y)$, in \eqref{BR-equ-2.6}, is nonnegative for any $(x,y,t)\in\Omega\times\Omega\times(0,+\infty).$
Define the local Green function with the Dirichlet boundary $G^\Omega:\Omega\times\Omega\to [0,\infty]$ as
\begin{equation}
	\label{BR-equ-3.3}
	G^\Omega(x,y):=\int_0^\infty p_t^\Omega(x,y){\rm d}t,\qquad \forall (x,y)\in\Omega\times\Omega.
\end{equation}

The Green function is the fundamental solution of the Laplace equation in the following sense.
\begin{lemma}\label{BR-lem-3.6}
 For any $f\in L^2(\Omega)$, we have
$$G[f](x):=\int_\Omega  G^\Omega(x,y)f(y){\rm d}\meas(y)\in D(\Delta_\Omega)$$
 and
$\Delta_\Omega G[f]=-f.$
 	\end{lemma}

\begin{proof} We first show $G[f]\in L^2(\Omega)$. Noticing that
$$G[f](x) =\int_\Omega\int_0^\infty p_t^\Omega(x,y)f(y){\rm d}t{\rm d}\meas(y)=\int_0^\infty H^\Omega_tf(x){\rm d}t,  $$  we have
\begin{equation*}
	\begin{split}
		\int_\Omega (G[f])^2{\rm d}\meas &=\int_\Omega\left( \int_0^\infty H^\Omega_tf(x){\rm d}t\right)^2\meas(x) \\
		&=\int_\Omega\left( \int_0^\infty H^\Omega_tf(x) e^{\lambda_1 t/2}\cdot e^{-\lambda_1t/2} {\rm d}t\right)^2\meas(x) \\
		&\ls\frac{1}{\lambda_1}\int_0^\infty e^{\lambda_1t}\|H^\Omega_tf\|^2_{L^2(\Omega)} {\rm d}t\ls \frac{\|f\|^2_{L^2(\Omega)}}{\lambda_1^2},
		 \end{split}
\end{equation*}
where we have used the H\"older inequality and the $L^2$-contraction (\ref{BR-equ-2.7}) in the last inequality.

Now we show that $G[f]\in D(\Delta_\Omega)$ and $\Delta_\Omega(G[f])=-f.$ In fact, for any $s>0$,
\begin{equation*}
	\begin{split}
		H^\Omega_s(G[f])(x)-G[f](x)&= \int_\Omega p_s(x,y)\int_0^\infty H^\Omega_tf(y){\rm d}t{\rm d}\meas(y)-\int_0^\infty H^\Omega_tf(x){\rm d}t\\
		&=\int_0^\infty  H^\Omega_s\circ H^\Omega_t(f) (x){\rm d}t-\int_0^\infty H^\Omega_tf(x){\rm d}t\\		&=\int_0^\infty  H^\Omega_{s+t} f (x){\rm d}t-\int_0^\infty H^\Omega_tf(x){\rm d}=-\int_0^s H^\Omega_tf(x){\rm d}t.	\end{split}
\end{equation*}
Therefore, we have
$$		\frac{H^\Omega_s(G[f])-G[f]}{s} =-\frac{1}{s}\int_0^s H^\Omega_tf{\rm d}t \to f \quad {\rm in}\ \ L^2(\Omega) $$
 as $s\to0$. It implies $G[f]\in D(\Delta_\Omega)$ and $\Delta_\Omega(G[f])=-f,$ and finishes the proof.
  \end{proof}

\begin{remark} \label{BR-remark-3.7}
	In \cite{BS20}, it was proved that any global Green function on $X$ (if it exists) is in $W^{1,1}(X)$.
	
	(2) It can also  be proved that $G(\cdot, x)\in L^1(\Omega)$ and ${\bf \Delta} G(x,\cdot)=-\delta_x(\cdot)$ in the sense of distrabutions.  We don't use this property in this paper.
\end{remark}

To get the upper bound of $G^\Omega$, we need the following large-time behavior for heat kernels, which is somewhat known to experts.

\begin{lemma}
	\label{BR-lem-3.8}
For each $\widehat T>0$  it holds
\begin{equation}\label{BR-equ-3.4}
p^\Omega_t(x,y)\ls  e^{-\lambda_1^\Omega(t-\widehat T)} \sqrt{p^\Omega_{\widehat T}(x,x)\cdot p^\Omega_{\widehat T}(y,y)},\end{equation}
 for all $(x,y)\in \Omega\times \Omega$ and any $ t\gs \widehat T$.
			\end{lemma} 			

\begin{proof}

From the fact that $p^\Omega_{t_0}(x,\cdot)\in D(\Delta_\Omega)$  for any $t_0>0$ and the $L^2$-contraction of $H^\Omega_t$, the Lebesgue's dominated convergence theorem implies
\begin{equation}\label{BR-equ-3.5}
	\frac{\partial}{\partial t}\int_\Omega p^\Omega_t(x,y)f(y){\rm d}\meas(y)=\int_\Omega\frac{\partial}{\partial t}p_t^\Omega(x,y)f(y){\rm d}\meas(y)
\end{equation}
for any $f\in L^2(\Omega)$ and $t>0$.

From the semi-group property of heat kernel  and (\ref{BR-equ-3.5}), we have for any $t>0$ and $x\in\Omega$ that
		\begin{equation*}
		\begin{split}
			\frac{\partial}{\partial t} p^\Omega_{2t}(x,x)&=\frac{\partial}{\partial t}\int_\Omega [p_t^\Omega(x,y)]^2{\rm d}\meas=2\int_\Omega \Delta_\Omega p_t^\Omega(x,y)p^\Omega_t(y,x){\rm d}\meas\\
			&=-2\int_\Omega |\nabla p_t^\Omega(x,y)|^2{\rm d}\meas	\\
			&\ls -2\lambda^\Omega_1 \cdot \int_\Omega [ p_t^\Omega(x,y)]^2{\rm d}\meas=-2\lambda^\Omega_1 \cdot  p^\Omega_{2t}(x,x).
						\end{split}
	\end{equation*} 	
This is, by putting $s=2t$, 	\begin{equation*}
					\frac{\partial}{\partial s} \ln p^\Omega_{s}(x,x)\ls - \lambda^\Omega_1.	
	\end{equation*}
					Therefore,  for each $\widehat T>0$ and $t>\widehat T$, by
					 integrating over $(\widehat T, t)$,  we have
\begin{equation*}
	p^\Omega_t(x,x) \ls e^{-\lambda_1^\Omega(t-\widehat T)}\cdot p^\Omega_{\widehat T}(x,x).			 			
\end{equation*}	
This is (\ref{BR-equ-3.4}) for $y=x$.
For any $x,y\in \Omega$, by \eqref{BR-equ-2.6}, we have
\begin{equation*}
\begin{split}	
 p^\Omega_t(x,y)&=\sum_je^{-\lambda_j^\Omega t }\phi_j^\Omega(x)\phi^\Omega_j(y)\ls \sum_je^{-\lambda_j^\Omega t/2 }\phi_j^\Omega(x)\cdot e^{-\lambda_j^\Omega t/2}\phi^\Omega_j(y)\\
 &\ls \sqrt{p_t^\Omega(x,x)\cdot p_t^\Omega(y,y)},
\end{split}	 			
\end{equation*}	
which proves (\ref{BR-equ-3.4}) for any $(x,y)\in\Omega\times\Omega$.
\end{proof}

Now, we will combine Theorem \ref{BR-thm-1.10} to get the following upper bound of $G^\Omega(x,y)$.\begin{theorem}
	\label{BR-thm-3.9}
	There exists a constant $c_3:=c_{3}(N,K,{\rm diam}(\Omega), \meas(\Omega), R_{\rm ext})>0$  such that
		\begin{equation}\label{BR-equ-3.6}
  G^\Omega(x,y)\ls c_3 \cdot \delta(x)\delta(y) \left(  \int^{D^2}_{d^2(x,y)} \frac{    {\rm d}t }{t\cdot \meas(B_{ \sqrt{t}}(x))}+1\right),
		\end{equation}
for any $(x,y)\in \Omega\times\Omega$, where $D:=\max\{R_{\rm ext}, {\rm diam}(\Omega)\}$.
 			\end{theorem}
			
\begin{proof}

 Denote by $r:=d(x,y)$, we have
  \begin{align*}
  	 G^\Omega(x,y)=	\int_0^{\frac{r^4}{D^2}}p^\Omega_t(x,y){\rm d}t+\int_{\frac{r^4}{D^2}}^{r^2}p^\Omega_t(x,y){\rm d}t+ \int_{r^2}^{ D^2 }p^\Omega_t(x,y){\rm d}t+\int_{D^2}^{ \infty }p^\Omega_t(x,y){\rm d}t.
  \end{align*}
We claim that there exists $c:=c_{K, N, D, {\rm diam}(\Omega), \meas(\Omega), R_{\rm ext}}>0$ such that

\begin{equation}\label{BR-equation-claim-1}
	\int_{D^2}^\infty p^\Omega_t(x,y){\rm d}t\ls c\cdot \delta(x)\delta(y),
\end{equation}

\begin{equation}\label{BR-equation-claim-2}
	\int_{r^2}^{D^2} p^\Omega_t(x,y){\rm d}t\ls  c\cdot  \delta(x)\delta(y)  \int_{r^2}^{D^2}\frac{{\rm d}t}{t\cdot \meas(B_{\sqrt{ t}}(y))},
\end{equation}

\begin{equation}
	\label{BR-equation-claim-3}
	\int_{\frac{r^4}{D^2}}^{r^2}p^\Omega_t(x,y){\rm d}t\ls  c\cdot \delta(x)\delta(y)   \int_{r^2}^{D^2}\frac{{\rm d}t}{t\cdot \meas(B_{\sqrt{ t}}(y))},
	\end{equation}
and
\begin{equation}\label{BR-equation-claim-4}
	\int_{0}^{ \frac{r^4}{D^2} }p^\Omega_t(x,y){\rm d}t \ls  c\cdot   \delta(x)\delta(y).
\end{equation}

By \eqref{BR-equ-1.6} and \eqref{BR-equ-3.4} (and by taking $\widehat T=D^2$ therein), and integrating over $(D^2,\infty)$, we have
$$\int_{D^2}^\infty p_t^\Omega(x,y)\ls c\delta(x)\delta(y)\int_{D^2}^\infty  e^{-\lambda_1^\Omega(t-D^2)} {\rm d}t
\ls c\delta(x)\delta(y)/\lambda_1^\Omega,$$
where we haved used $\meas(B_D(y))\gs \meas(\Omega)$ for any $y\in\Omega$ (since $\Omega\subset B_D(y)$). This is \eqref{BR-equation-claim-1}, because $\Omega\subset B_D(y)$ also implies
$\lambda^\Omega\gs  \lambda^{B_D(y)}\gs c_{N,K,D}.$

By using \eqref{BR-equ-1.6} again, we have for any $t\in(0,D^2)$ that
\begin{equation}\label{BR-equ-3.11}
	p^\Omega_t(x,y)\ls \frac{\delta(x)\delta(y)}{t} \cdot\frac{C_2}{\meas(B_{\sqrt{ t}}(y))}\cdot\exp\left(-\frac{r^2}{C_1t}\right)\ls \frac{\delta(x)\delta(y)}{t} \cdot\frac{C_2}{\meas(B_{\sqrt{ t}}(y))}.
\end{equation}
Integrating it over $(r^2, D^2)$, we get \eqref{BR-equation-claim-2}.

Integrating (\ref{BR-equ-3.11}) over $(\frac{r^4}{D^2}, r^2)$ and setting $s=\frac{{r^4}}{t},$ we have
\begin{equation*}	
\begin{split}
  \int_{\frac{r^4}{D^2}}^{r^2}p^\Omega_t(x,y){\rm d}t&\ls C_2\delta(x)\delta(y) \int^{r_2}_{\frac{r^4}{D^2}}\frac{\exp(-\frac{r^2}{C_1t})}{t\cdot \meas(B_{\sqrt t}(x))}{\rm d}t\\
  &=C_2 \delta(x)\delta(y)\int^{D^2}_{r^2}\frac{r^4}{s^2\cdot t}\cdot\frac{\exp(-\frac{s}{C_1r^2})}{\meas(B_{ \sqrt{t}}(x))}{\rm d}s.
  \end{split}
\end{equation*}
 By (\ref{BR-equ-2.4}) and $t=r^4/s\ls r^2\ls s\ls D^2$,
 $$\frac{\meas(B_{\sqrt s}(x))}{\meas(B_{\sqrt t}(x))}\ls  c\left(\frac{s}{r^2}\right)^N.$$
Therefore, we have
\begin{equation*}
	\begin{split}
		\int^{D^2}_{r^2}\frac{r^4}{s^2\cdot t}\cdot\frac{\exp(-\frac{s}{C_1r^2})}{\meas(B_{ \sqrt{t}}(x))}{\rm d}s
				&\ls C_2\int^{D^2}_{r^2}\frac{1}{s}\cdot\frac{\exp(-\frac{s}{C_1r^2})}{\meas(B_{ \sqrt{s}}(x))} \cdot \left(\frac{s}{r^2}\right)^N {\rm d}s\\
		&\ls c\int^{D^2}_{r^2} \frac{ {\rm d}s }{s\cdot\meas(B_{ \sqrt{s}}(x))},		
					\end{split}
\end{equation*}
where we used that $\sup_{a:=s/r^2\gs 1}a^{N}\exp(-a/C_1):=C(N,C_1)<+\infty,$ which follows \eqref{BR-equation-claim-3}.

Integrating (\ref{BR-equ-3.11}) over $(0,\frac{r^4}{D^2})$ and setting $s=\frac{{r^4}}{t},$ we have
\begin{equation*}
	\begin{split}
				\int_0^{\frac{r^4}{D^2}}p^\Omega_t(x,y){\rm d}t& \ls C_2\delta(x)\delta(y) \int_0^{\frac{r^4}{D^2}}\frac{\exp(-\frac{r^2}{C_1t})}{t\cdot \meas(B_{\sqrt t}(x))}{\rm d}t\\
		&=C_2 \delta(x)\delta(y) \int_{D^2}^{\infty}\frac{1}{s }\cdot\frac{\exp(-\frac{s}{C_1r^2})}{\meas(B_{ \sqrt{t}}(x))}{\rm d}s.
			\end{split}
\end{equation*}
Since $\sqrt s\gs D\gs {\rm diam}(\Omega)$, by the Bishop-Gromov inequality (\ref{BR-equ-2.3}) and $t\ls s$ (since $t\ls r^4/D^2\ls r^2$), we get
 $$  \frac{1}{\meas(B_{\sqrt t}(x))}\ls \frac{1}{\meas(\Omega)}\frac{  \meas(B_{\sqrt s}(x))}{\meas(B_{\sqrt t}(x))}\ls C\frac{  V_{N,K}(\sqrt s)}{V_{N,K}(\sqrt t)}. $$
Since $K<0$ and $N>1$ we have assumed, we have
 $V_{N,K}(\sqrt t)\gs c_Nt^{N/2}$
  and
$$V_{N,K}(\sqrt s)= \int_0^{\sqrt s}\sinh^{N-1}_{\frac{-K}{N-1}}(\tau){\rm d}\tau\ls \int_0^{\sqrt s}\left(\frac{\exp(\sqrt{\frac{-K}{N-1}} \ \tau)}{\sqrt{\frac{-K}{N-1}}}\right)^{N-1}{\rm d}\tau \ls C\cdot e^{\sqrt{-K(N-1)}\sqrt s}.$$
 Therefore, we obtain
 \begin{equation*}
 \begin{split}
	   \frac{1}{s } \cdot\frac{V_{N,K}(\sqrt s)}{V_{N,K}(\sqrt t)}&
	    \ls C\cdot \exp\left( \sqrt{-K(N-1)}\sqrt s\right)\cdot \frac{s^{N/2-1}}{r^{2N}}\\
	    & \ls C\cdot \exp\left(\frac{s}{2C_1r^2}+ \frac{-C_1K(N-1)}{2} D^2\right)\cdot \left(\frac{s}{r^2}\right)^{N}\cdot D^{-N-2}\\
	    &=  C'\cdot \exp\left(\frac{s}{2C_1r^2}\right)\cdot \left(\frac{s}{r^2}\right)^{N},	
	       \end{split}	
			 \end{equation*}
 where we have used $\sqrt{-K(N-1)}\sqrt s\ls \frac{s}{2C_1r^2}+\frac{-C_1K(N-1)r^2}{2}$, $r\ls D$
  and $ s^{-N/2-1}\ls D^{-N-2}$.
By integrating over $(D^2,\infty)$ , we conclude that
\begin{equation*}
	\begin{split}
	\int_0^{\frac{r^4}{D^2}} p^\Omega_t(x,y){\rm d}t&\ls  C\cdot  \delta(x)\delta(y) \int_{D^2}^\infty \exp\left(-\frac{s}{C_1r^2}+\frac{s}{2C_1r^2}\right)\cdot \left(\frac{s}{r^2}\right)^N  {\rm d}s\\
	&=  C\cdot   \delta(x)\delta(y) \cdot r^2   \int_{D^2/r^2}^\infty \exp\left(-\frac{\tau}{2C_1}\right)\tau^N {\rm d}\tau  \\
	&\ls  C D^2  \cdot\delta(x)\delta(y)\cdot   \int_{0}^\infty \exp\left(-\frac{\tau}{2C_1}\right)\tau^N {\rm d}\tau,
	  				\end{split}
\end{equation*}
which implies  the claim \eqref{BR-equation-claim-4}, and hence we have finished the proof of   \eqref{BR-equ-3.6}.
\end{proof}

Now we want to estimate $|\nabla G^\Omega|$. Notice that $G^\Omega(x,y)$ is not in $W^{1,2}$ in general (see Remark  \ref{BR-remark-3.7} (1)); for this reason, we consider the following approximations
$$G_{\varepsilon}(x,y):=\int_\varepsilon^\infty p_t^\Omega(x,y){\rm d}t,\qquad G_{\varepsilon,T}(x,y):=\int_\varepsilon^T p_t^\Omega(x,y){\rm d}t$$
 for any $0<\varepsilon<T<\infty$.
\begin{lemma}
	\label{BR-lem-3.10}
	For any $\varepsilon>0, T\in(\varepsilon,\infty)$ and any given $x\in\Omega$, it holds
$G_\varepsilon(x,\cdot)\in D(\Delta_\Omega)\cap Lip_{\rm loc}(\Omega)$  and $G_{\varepsilon,T}(x,\cdot)\in Lip_{\rm loc}(\Omega).$
 Moreover, we have
 	\begin{equation}\label{BR-equ-3.12}
			|\nabla G_{\varepsilon,T}(x,\cdot)|^2(y)\ls    \left( 2p^\Omega_{T}(x,y)  + C_{N,K,{\rm diam}(\Omega)} \cdot \frac{G_\varepsilon(x,y)}{\delta^2(y)} \right)\cdot G_\varepsilon(x,y) 	
	\end{equation}	
and
	\begin{equation}\label{BR-equ-3.13}
		|\nabla G_{\varepsilon}(x,\cdot)|(y)\ls  C_{N,K,{\rm diam}(\Omega)}\cdot  \frac{G^\Omega(x,y)}{\delta(y)}
	\end{equation}
	for almost all  $y\in \Omega$ such that $\delta(y)\ls \sqrt{\varepsilon}$.
\end{lemma}

\begin{proof}
(i) By the definition, for any fixed $x\in\Omega$ and $\varepsilon>0$, we have
\begin{equation*}
	\begin{split}	
	G_\varepsilon(x,y)&=\int_\varepsilon^\infty p^\Omega_{t}(x,y){\rm d}t=\int_0^\infty p^\Omega_{t+\varepsilon}(x,y){\rm d}t\\
	&=\int_0^\infty\int_\Omega p^\Omega_\varepsilon(x,z)p^\Omega_t(z,y){\rm d}\meas(z){\rm d}t=G[p^\Omega_\varepsilon(x,\cdot)](y).		
	\end{split}
\end{equation*}
From Lemma \ref{BR-lem-3.6}, $G_\varepsilon(x,\cdot)\in D(\Delta_\Omega)$ and $\Delta_\Omega G_\varepsilon(x,\cdot)=-p^\Omega_\varepsilon(x,\cdot).$ According to Lemma \ref{BR-lem-2.5}, we have $G_\varepsilon(x,\cdot)\in Lip_{\rm loc}(\Omega).$

(ii) Given any  $\varepsilon>0, T\in(\varepsilon,\infty)$ and  $x\in\Omega$, we  claim that $G_{\varepsilon,T}(x,\cdot)\in Lip_{\rm loc}(\Omega)$ and
\begin{equation}\label{BR-equ-3.14}
|\nabla  G_{\varepsilon,T}(x,\cdot)|(y)\ls \int_\varepsilon^T|\nabla p_t^\Omega(x,\cdot)|(y){\rm d}t.
\end{equation}
for almost all  $y\in\Omega$.

Because of $p^\Omega_t(x,\cdot)\in Lip_{\rm loc}(\Omega\times (0,T))$ (by \cite[Corollary 1.5]{HZ20}) and $p^\Omega_t(x,\cdot)>0$ on $\Omega\times (0,\infty)$, for any ball $B\Subset\Omega$ containing $y$  and $[\varepsilon,T]\subset(0,\infty)$,  we have  $m\ls p^\Omega_t(x,\cdot)\ls M$ on $\overline{B} \times [\varepsilon,T]$ for some $m>0$ and $M<\infty$ (depending on $B, \varepsilon$ and $T$). According to \cite[Theorem 1.1]{HZ20}, we know that
$$|\nabla p^\Omega_t(x,\cdot)|\ls C_{N, K, B, T, m, M}\quad {\rm almost\ everywhere\ in}\  \  \overline{B} \times [\varepsilon, T].$$
  By  using Lebesgue's dominated convergence theorem and
$$\frac{|G_{\varepsilon,T}(x,y')-G_{\varepsilon,T}(x,y'')|}{d(y',y'')}\ls\int_\varepsilon^T \frac{|p^\Omega_{t}(x,y')-p^\Omega_{t}(x,y'')|}{d(y',y'')}{\rm d}t$$
for any $y',y''\in B$, $y'\not=y''$,
we get  $G_{\varepsilon,T}(x,\cdot)\in Lip(B)$ and 	\begin{equation*}
	{\rm Lip} G_{\varepsilon,T}(x,\cdot)(y)\ls \int_\varepsilon^T {\rm Lip}p_t^\Omega(x,\cdot) (y){\rm d}t.
\end{equation*}
for any $y\in\Omega$. 	
	it follows the claim (\ref{BR-equ-3.14}).
	
(iii) It suffices to show  that (\ref{BR-equ-3.12}) and (\ref{BR-equ-3.13})  hold almost eveywhere in  $ B_R(y_0)$ for any  $y_0\in \Omega$   such that $\delta(y_0)<\sqrt\varepsilon$ and $R:=\delta(y_0)/2.$

	Letting $f(\cdot,t):=\ln p^\Omega_t(x,\cdot)$, by local Li-Yau's gradient estimate \cite[Theorem 1.4]{ZZ16}, we have   that
	$$|\nabla f|^2(y,t)-2\frac{\partial f}{\partial t}(y,t)\ls C_{K,N}\left(\frac{1}{R^2}+\frac{1}{t}+1\right)$$
	for almost all $y\in B_R(y_0)$ and almost all $t\in(0,\infty)$. Therefore,  for $t>\varepsilon,$
	$$\frac{|\nabla p^\Omega_t(x,\cdot)|^2(y)}{p^\Omega_t(x,y)}\ls 2\frac{\partial}{\partial t} p^\Omega_t(x,y)+C_{N,K} \left(\frac{1}{R^2}+\frac{1}{\varepsilon}+1\right)p^\Omega_t(x,y)$$	
	for almost all $y\in B_R(y_0)$. Integrating over $[\varepsilon,T]$ and by H\"older inequality, we get
	\begin{equation*}
		\begin{split}
			\left(\int_{\varepsilon}^{T}|\nabla p^\Omega_t(x,\cdot)|(y){\rm d}t\right)^2& \ls  \int_{\varepsilon}^{T}\frac{|\nabla p^\Omega_t(x,\cdot)|^2(y)}{p^\Omega_t(x,y)}{\rm d}t\cdot \int_{\varepsilon}^{T}  p^\Omega_t(x,y){\rm d}t\\
			&\ls \left[2\left(p^\Omega_{T} -p^\Omega_{\varepsilon} \right)+C_{N,K}\left(\frac{1}{R^2}+\frac{1}{\varepsilon}+1\right)G_{\varepsilon,T} \right]\cdot G_{\varepsilon,T}(x,y)		\end{split}
	\end{equation*}
for almost all $y\in B_R(y_0)$. By using (\ref{BR-equ-3.14}), $p^\Omega_\varepsilon(x,y)\gs0$ and $\varepsilon\gs 4 R^2$ (since $2R=\delta(y_0)\ls \sqrt\varepsilon$), we get
		\begin{equation}\label{BR-equ-3.15}
		\begin{split}
			|\nabla G_{\varepsilon,T}(x,\cdot)|^2(y)&\ls    \left( 2p^\Omega_{T}(x,y)  + C_{N,K,{\rm diam}(\Omega)} \cdot \frac{G_{\varepsilon,T}(x,y)}{R^2} \right)\cdot G_{\varepsilon,T}(x,y) \\
			&\ls    \left( 2p^\Omega_{T}(x,y)  + C_{N,K,{\rm diam}(\Omega)} \cdot \frac{G_{\varepsilon}(x,y)}{R^2} \right)\cdot G_{\varepsilon}(x,y) 	
		\end{split}
			\end{equation}	
for almost all $y\in B_R(y_0)$, since $G_{\varepsilon, T}(x,y)\ls G_\varepsilon(x,y)$,
which implies (\ref{BR-equ-3.12}), by using $\delta(y)\ls \delta(y_0)+R=3R$.

From \eqref{BR-equ-3.4}, we have
$$\lim_{T\to\infty}p^\Omega_T(x,y)=0.$$
By using (\ref{BR-equ-3.4}) again (taking  $\widehat T=\varepsilon$), and then integrating over $(\varepsilon,+\infty)$, combining with (\ref{BR-equ-2.8}), we get
$$ G_{\varepsilon}(x,y)\ls \frac{1}{\lambda^\Omega_1}\sqrt{p^\Omega_\varepsilon(x,x)\cdot p^\Omega_\varepsilon(y,y)}\ls C_{\varepsilon, N,K,\Omega}.$$
Letting $T\to\infty$ in \eqref{BR-equ-3.15}, using Arzela-Ascoli theorem and $G_\varepsilon\ls G^\Omega$, we conclude (\ref{BR-equ-3.13}). The proof is finished.
\end{proof}

\section{A version of the Gauss-Green formula on sets with one-sided approximations}

As we have described in the introduction, the proof of Theorem \ref{BR-thm-1.8-main} is based on  Lemma \ref{BR-lem-3.10} and Proposition \ref{BR-prop-1.12}, which will be proved in this section.
Let $(X,d,\meas)$ be an $RCD(K,N)$ space with $K\in\mathbb R$ and $N\in(1,\infty)$.

\subsection{The Gauss-Green formula on boundary-regular domains}
 Let us begin with a simple case when $\Omega$ is a boundary-regular set, which will be defined in the following.
\begin{defn}
	\label{BR-def-4.2}
	Let $\Omega$ be an open subset of $X$ with $\meas(\Omega)<\infty$. The {\emph{ upper}} and {\emph{lower inner Minkowski contents}} are given by
	$$ \mathscr{IM}_+(\Omega):=\limsup_{\delta\to0^+}\frac{\meas (\Omega\setminus\Omega_\delta)}{\delta},\quad  \mathscr{IM}_-(\Omega):=\liminf_{\delta\to0^+}\frac{\meas (\Omega\setminus\Omega_\delta)}{\delta},$$
	recalling $\Omega_\delta=\{x\in \Omega|\ d(x,\partial\Omega)>\delta\} $ for any $\delta\in(0,1)$. If $\mathscr{IM}_+(\Omega)=\mathscr{IM}_-(\Omega)$, we call $\mathscr{IM} (\Omega):=\mathscr{IM}_+(\Omega)$ the {\emph{inner Minkowski contents}}  of $\Omega$.
	\end{defn}

\begin{defn}[cf. \cite{BCM22,MMS16}] \label{BR-def-4.3}
	Let $\Omega$ be an open set of finite perimeter, $\meas(\Omega)<\infty$ and $\meas(\partial \Omega)=0$.
It is called \emph{boundary-regular} if
$$\mathscr{IM}(\Omega)=    Per_\Omega(X).$$
	  	 		\end{defn}

\begin{remark}\label{BR-rem-4.4}
(1) The definition of the boundary-regularity can be replaced to  require
	$$Per_\Omega(X)\gs \mathscr{IM}_{+}(\Omega),$$
because the definition of perimeter implies the inverse inequality
	$$Per_\Omega(X)\ls \liminf_{\delta\to0}\int_X|\nabla \eta_\delta|{\rm d}\meas=\liminf_{\delta\to0}\frac{\meas\big(\Omega\setminus\Omega_\delta\big)}{\delta}=\mathscr{IM}_{-}(\Omega)$$
where
$$\eta_\delta(x):=\begin{cases}
	1 & x\in \Omega_\delta,\\
	\frac{d(x,X\setminus\Omega)}{\delta} & x\in \Omega\setminus\Omega_\delta,\\
	0& x\in X\setminus\Omega.
\end{cases}
	$$
		
	(2)  Given any $x_0\in X$, it was proved in \cite{MMS16,GH16} that for almost every $r\in(0,\infty)$ the ball $B_r(x_0)$  is a boundary-regular domain (see also \cite[Remark 4.9]{BCM22}).
	\end{remark}

 We have the following approximations of a boundary-regular set from its interior.
 \begin{lemma}\label{BR-lem-4.5}
	Suppose that the open set $\Omega$ is bounded and boundary-regular. Then there is a sequence $\phi_k\in Lip_0(\Omega)$, $0\ls \phi_k\ls 1$, such that   $\phi_k \to \chi_\Omega$ in $L^1(X)$,  and
	$$|\nabla \phi_k|\cdot\meas \rightharpoonup Per_{\Omega}\quad {\rm weakly\ as}\ k\to\infty.$$
\end{lemma}
\begin{proof}
By the lower semi-continuity of $Per_\Omega$ in $L^1(X)$, it suffices to show that there is a sequence $\phi_k\in Lip_0(\Omega)$, $0\ls \phi_k\ls 1$, such that   $\phi_k \to \chi_\Omega$  in $L^1(X)$,   and
\begin{equation}\label{BR-equ-4.3}
 	\limsup_{k\to\infty}\int_X|\nabla \phi_k|{\rm d}\meas \ls Per_\Omega(X).
\end{equation}

  Let $\delta\in(0,1)$ and consider the function
\begin{equation*}
	\phi^\Omega_\delta(x):=
	\begin{cases}
		0 & x\in X\setminus\Omega_{\delta^2} ,\\
		\frac{d(x, X\setminus  \Omega_{\delta^2})}{\delta-\delta^2}, & x\in\Omega_{\delta^2}\setminus (\Omega_{\delta^2})_{\delta-\delta^2},\\
		1,\ & x\in(\Omega_{\delta^2})_{\delta-\delta^2}.
	\end{cases}
\end{equation*}
We have that  $\phi^\Omega_\delta\in Lip_0(\Omega)$  and
$\phi^\Omega_\delta(x)\to 1$ for each $x\in\Omega$, as $\delta\to0$. Hence $\phi^\Omega_\delta\to \chi_\Omega$ in $L^1(X)$.
By noticing $(\Omega_{\delta^2})_{\delta-\delta^2}\supset\Omega_\delta$,
$$  \limsup_{\delta\to0}\int_\Omega|\nabla \phi^\Omega_\delta|{\rm d}\meas\ls \limsup_{\delta\to0}\frac{\meas(\Omega\setminus\Omega_\delta)}{\delta-\delta^2}=\limsup_{\delta\to0}\frac{\meas(\Omega\setminus\Omega_\delta)}{\delta}.$$
It implies  (\ref{BR-equ-4.3}), by taking $\phi_k:=\phi^\Omega_{1/k}$.
\end{proof}

We now give the following version of the Gauss-Green formula for the boundary-regular domains, via an approximating argument.

\begin{proposition}
	\label{BR-prop-4.6}
 Let  $\Omega$ be a bounded, boundary-regular and open domain. Let $f\in W^{1,2}(\Omega)\cap C(\overline\Omega)$, $f\gs0$.  Suppose that $g\in W^{1,2}(\Omega)$ such that  ${\bf \Delta }g  $ is a Radon measure on $\Omega$ such that ${\bf \Delta}g\ls a\cdot  \meas$ for some function $a(x)\in L^1(\Omega)$. Assume $|\nabla g|(x)\ls h(x)$ $\meas$-a.e. $x\in \Omega$ for some upper semicontinuous function $h$  on $\overline\Omega$. Then it holds
 \begin{equation}\label{BR-equ-4.4}
 	-\int_\Omega fa{\rm d}\meas\ls \int_\Omega\ip{\nabla f}{\nabla g}{\rm d}\meas+\int_{\partial\Omega}fh{\rm d}Per_\Omega.
 \end{equation} 	
\end{proposition}

\begin{proof}
	Let $\phi_k$ be a sequence in $Lip_0(\Omega)$ given in Lemma \ref{BR-lem-4.5}. Then $\phi_k f\in W^{1,2}_0(\Omega)$. Since $g\in W^{1,2}(\Omega)$, from the definition of ${\bf \Delta} g$, we have
	\begin{equation}\label{BR-equ-4.5}
		\begin{split}
			-\int_\Omega\phi_k f{\rm d}{\bf\Delta}g&=\int_\Omega\ip{\nabla(f\phi_k)}{\nabla g}{\rm d}\meas\\
			&=\int_\Omega \phi_k\ip{\nabla f}{\nabla g}{\rm d}\meas+ \int_\Omega f\ip{\nabla \phi_k }{\nabla g}{\rm d}\meas	\\
			&\ls \int_\Omega \phi_k\ip{\nabla f}{\nabla g}{\rm d}\meas+ \int_\Omega fh\cdot|\nabla \phi_k |  {\rm d}\meas,	
					\end{split}
	\end{equation}
	where, in the last inequality, we have used $f\gs 0$ and $|\nabla g|\ls h$ a.e. in $\Omega$.

	The fact $\phi_k\to \chi_\Omega$ in $L^1(X)$ implies that there exists a subsequence of $\phi_k$, denoted by $\phi_k$ again, such that $\phi_k(x)\to \chi_\Omega(x)$ for $\meas$-almost all $x\in X$. We have $|\phi_k f|\ls \max_{\overline\Omega}|f|$ for each $k$,  and the measure ${\bf \Delta}g\ls a\cdot \meas$ on $\Omega$. By Lebesgue's dominated convergence theorem, we get
	\begin{equation*}
		\lim_{k\to\infty}\int_{\Omega} \phi_k f{\rm d}{\bf \Delta}g\ls \lim_{k\to\infty}\int_{\Omega} \phi_k fa{\rm d}\meas= \int_\Omega \lim_{k\to\infty}(\phi_k fa){\rm d}\meas=\int_\Omega  fa{\rm d}\meas.	\end{equation*}
	Considering the right-hand side of \eqref{BR-equ-4.5}, since $0\ls \phi_k\ls 1$ and $\ip{\nabla f}{\nabla g}\in L^1(\Omega)$, Lesbegue dominated converge theorem yields $$\int_\Omega\phi_k\ip{\nabla f}{\nabla g}{\rm d}\meas\to \int_\Omega\ip{\nabla f}{\nabla g}{\rm d}\meas$$ as $k\to \infty.$ At last by using $ fh\gs0 $,  the assumption that $fh$ is upper semicontinuous on $\overline\Omega$, and  Lemma \ref{BR-lem-4.5},
 	we have
 $$\limsup_{k\to\infty}\int_\Omega fh|\nabla \phi_k|{\rm d}\meas=\limsup_{k\to\infty}\int_{\overline\Omega}fh|\nabla \phi_k|{\rm d}\meas\ls \int_{\partial \Omega}fh{\rm d}Per_\Omega,$$ where we have used the fact that $Per_\Omega$ is supported in $\partial\Omega.$ Now letting $k\to\infty$ in (\ref{BR-equ-4.5}),  the desired estimate \eqref{BR-equ-4.4} follows.
\end{proof}

\begin{remark}
If $g$ can be extended to a function $g^*$ on $X$ such that ${\bf \Delta}g^*$ is a finite Radon measure on $X$, then the assumption ${\bf \Delta}g\ls a\cdot \meas$ for some $a\in L^1(\Omega)$ is not necessary. Indeed, in this case, since $\phi_k\to \chi_\Omega$ for {\emph{each}} $x\in X$, the Lebugues dominated convergence theorem states
$$\lim_{k\to\infty}\int_\Omega\phi_kf{\rm d}{\bf \Delta}g=\lim_{k\to\infty}\int_X\phi_kf{\rm d}{\bf \Delta}g^*=\int_X f{\rm d}{\bf \Delta}g^*.$$
	The conlusion (\ref{BR-equ-4.4})
	will be replaced by
 	$$-\int_X f {\rm d}{\bf \Delta}g^*\ls \int_\Omega\ip{\nabla f}{\nabla g}{\rm d}\meas+\int_{\partial\Omega}fh{\rm d}Per_\Omega.$$
 	\end{remark}
	
Recall that the version of the Gauss-Green formula given in (\ref{equation-1-12})--(\ref{equation-1-13}) has been extended to the $RCD$ setting in \cite{BPS23-jems, BPS23, BCM22}. For convenience, here we state only the case where the vector field is a gradient field.
\begin{theorem}[\cite{BPS23}]\label{BR-thm-4.1}
Let $(X, d, \meas)$ be an $RCD(K,N)$ space for some $K\in \mathbb R$ and $N\in(1,\infty)$. Let $\Omega$ be a set of finite perimeter and let $g$ be a Lipschitz function on $X$ such that its distribution Laplacian ${\bf \Delta}g$ is a finite Radon measure on $X$. Then there exists a function  $(\nabla g,\nu_\Omega)_{\rm in}$   in  $L^\infty(\mathscr F\Omega, Per_\Omega)$  such that
\begin{align}
	\label{BR-equ-4.1}		
	-\int_{\Omega^{(1)}} \varphi {\rm d}{\bf \Delta}g & = \int_\Omega\ip{\nabla \varphi}{\nabla g}{\rm d}\meas+\int_{\mathscr F\Omega}\varphi(\nabla g,\nu_\Omega)_{\rm in}{\rm d}Per_\Omega
\end{align}
for any $\varphi\in Lip_0(X)$, and
		\begin{align}
	 \label{BR-equ-4.2}
		 \|(\nabla g,\nu_\Omega)_{\rm in}\|_{L^\infty(\mathscr F\Omega,Per_\Omega)}&\ls \||\nabla g|\|_{L^\infty(\Omega,\meas)}.
\end{align}					
\end{theorem}

Proposition \ref{BR-prop-4.6} requires that  ${\bf \Delta} g$ is a Radon measure on $\Omega$ merely,  after assuming that $\Omega$ is boundary-regular.   Nevertheless, Theorem  \ref{BR-thm-4.1}  requires that  ${\bf \Delta}g$ is a Radon measure on the whole space $X$ (without the assumption that $\Omega$ is boundary-regular). This suggests the following definition.

\begin{defn}\label{BR-def-4.7}
A bounded open domain  $\Omega\subset X$  is called a {$\mathcal{DM}^\infty$-\emph{extentable domain}} \footnote{The notion $\mathcal{DM}^\infty$ is based on the fact that $\nabla g$ is a bounded divergence-measure field.}
  if for each     $g\in Lip(\Omega)$ such that ${\bf\Delta}g$ is a finite Radon measure on $\Omega$, there exists     $g^*\in Lip(X)$ such that ${\bf \Delta}g^*$ is a finite Radon measure on $X$ and $g^*=g$ on $\Omega$. Such a function $g^*$ is called a {$\mathcal{DM}^\infty$-\emph{extension}} of $g$.
\end{defn}
Remark that if there exists an open domain $\Omega'\Supset \Omega$ and a function $g^*\in Lip(\Omega')$ such that ${\bf \Delta}g^*$ is a finite Radon measure on $\Omega'$ and $g^*=g$ on $\Omega$, then  $g$ has a  $\mathcal{DM}^\infty$-extension on $X$. Indeed, take any $\eta\in {\rm Test}^\infty_c(X)$ such that $\eta\equiv1$ on $\Omega$ and ${\rm supp}(\eta)\subset\Omega',$ then $g^*\cdot \eta$ is the desired extension of $g$ on $X$. Here $g^*\cdot \eta$ takes zero out of $\Omega'$.

 In the following, we want to investigate the geometric condition to ensure that an open set of finite perimeter is boundary-regular.

\begin{lemma}
	\label{BR-lem-4.9}
 Let $\Omega$ be a bounded open set of finite perimeter. If the function $\rho_{\rm e}(x):=d(x,X\setminus \Omega)$ has a $\mathcal{DM}^\infty$-extension, denoted by $g$, and if ${\bf \Delta}g\big(\partial\Omega\cap \Omega^{(1)} \big) =0$, then $\Omega$ is boundary-regular.
 \end{lemma}

\begin{proof}
Fix any $\delta\in(0,1)$, and put
$$f_\delta(x):=\begin{cases}
	0 & x\in \Omega_\delta,\\
	1-\frac{\rho_{\rm e}(x)}{\delta} & x\in \Omega\setminus\Omega_\delta,\\
	1& x\in X\setminus\Omega.
\end{cases}
	$$
Take a cut-off function $\phi\in {\rm Test}^\infty_c(X)$ such that $0\ls \phi(x)\ls 1$ in $X$ and that $\phi(x)\equiv 1$ in a bounded open subset $\Omega'\supset   \overline\Omega.$   Then $\phi f_\delta\in Lip_0(X),$  $\phi g \in Lip_0(X)$ and
$${\bf \Delta}(\phi g)=\big(g  \Delta\phi+2\ip{\nabla \phi}{\nabla g}\big)\cdot \meas+\phi\cdot {\bf \Delta}g,$$
which is a finite Radon measure on $X$ since ${\rm supp}(\phi)$ is compact and that $g\Delta\phi,\ip{\nabla \phi}{\nabla g} $ are in $L^\infty(X).$
By using the Gauss-Green formula (\ref{BR-equ-4.1}) to $\phi f_\delta $ and $\phi g$ on $\Omega$, we have
\begin{equation}\label{BR-equ-4.7}
	\begin{split}
		 -\int_{\Omega^{(1)}} \phi f_{\delta}\ {\rm d}{\bf  \Delta}(\phi g)  & = \int_\Omega\ip{\nabla (\phi f_\delta)}{\nabla (\phi g)}{\rm d}\meas +\int_{\mathscr F\Omega}\phi f_\delta (\nabla (\phi g),\nu)_{\rm in}{\rm d}Per_\Omega\\
		 &\ls  \int_\Omega\ip{\nabla   f_\delta }{\nabla   g }{\rm d}\meas + \int_{\mathscr F\Omega}|\phi f_\delta|\ {\rm d} Per_\Omega(X)\cdot  \|\nabla  g \|_{L^\infty(\Omega)} \\
		 	 &\ls -\frac{\meas(\Omega\setminus\Omega_\delta)}{\delta}  + Per_\Omega(X) ,
		 	 	\end{split}
\end{equation}	
where we have used
$$\ip{\nabla (\phi f_\delta)}{\nabla(\phi g)}=\ip{\nabla f_\delta}{\nabla g}=\frac{-\ip{\nabla \rho_{\rm e}}{\nabla \rho_{\rm e}} }{\delta}\cdot\chi_{\Omega\setminus\Omega_\delta}=\frac{-\chi_{\Omega\setminus\Omega_\delta}}{\delta}, \ \ \ \meas{\rm-a.e.\ in }\ \ \Omega $$
since $\phi\equiv1$ on $\Omega'\supset\overline \Omega$ and $g=\rho_{\rm e}$ on $\Omega$,
  $$\|(\nabla (\phi g),\nu)_{\rm in}  \|_{L^\infty(\partial\Omega, Per_\Omega)}\ls \|\nabla(\phi g)\|_{L^\infty(\Omega)}= \|\nabla  g \|_{L^\infty(\Omega)}=1$$
  by (\ref{BR-equ-4.2}), and used $|\phi f_\delta |\ls 1 $ on $X$.

 We now consider the left-hand side of (\ref{BR-equ-4.7}). Since $\phi\equiv1$ on $\Omega'$ and $\Omega^{(1)}\subset\overline\Omega\subset \Omega'$, we have the measure
 $\big({\bf \Delta}(\phi g)\big)|_{\Omega^{(1)}}= ({\bf \Delta} g)|_{\Omega^{(1)}}.$ Hence we get
  $$ \int_{\Omega^{(1)}} \phi f_{\delta}\ {\rm d}{\bf  \Delta}(\phi g) =   \int_{\Omega^{(1)}}  f_{\delta}\ {\rm d}{\bf  \Delta}g= \int_{\Omega^{(1)}\setminus\Omega}  f_{\delta}\ {\rm d}{\bf  \Delta}g+\int_{ \Omega}  f_{\delta}\ {\rm d}{\bf  \Delta}g,$$
  where we have used $\Omega\subset \Omega^{(1)}$ (since $\Omega$ is open).
 From the definition of $f_\delta$, we have
 $$\lim_{\delta\to0}   f_\delta(x)	=0\ \  {\rm for\ each}\ \ x\in\Omega.$$
 Since $({\bf \Delta}g)|_{\Omega}$ is finite and $|f_\delta|\ls1$ for each $\delta\in(0,1)$,  the Lebesgue's dominated convergence theorem implies
 \begin{equation}
 	\label{BR-equ-4.8}	
 	\lim_{\delta\to0}\int_{\Omega}   f_\delta\ {\rm d}{\bf\Delta}g =0.
 	\end{equation}
Notice that  $  f_\delta \equiv 1$ on $\partial\Omega\supset \Omega^{(1)}\setminus\Omega$  for each $\delta>0$, we have
$$   \int_{ \Omega^{(1)}\setminus\Omega}  f_\delta\ {\rm  d}{\bf\Delta}g= {\bf\Delta}g\big(\Omega^{(1)}\setminus\Omega\big).$$
 Combining these and  (\ref{BR-equ-4.7}),  then letting $\delta\to 0$ and substituting (\ref{BR-equ-4.8}),
it follows
 \begin{equation*}
	-{\bf\Delta}g\big(\Omega^{(1)}\setminus\Omega\big)+ \mathscr{IM}_+(\Omega)\ls    Per_\Omega(X).
		\end{equation*}	
Notice that $\partial\Omega\cap \Omega^{(1)}=\Omega^{(1)}\setminus \Omega$. The assumption ${\bf\Delta}g\big(\partial \Omega\cap \Omega^{(1)} \big)=0$ gives $\mathscr{IM}_+(\Omega)\ls    Per_\Omega(X)$, and hence $\Omega$ is boundary-regular. The proof is finished.
\end{proof}

		\begin{remark}\label{BR-rem-4.8}
	The existence of a $\mathcal{DM}^\infty$-extension of $g$ is a non-trivial condition. In the Euclidean setting $X=\mathbb R^N$, it was shown \cite[Theorem 4.2]{CLT20} that any open set $\Omega\subset \mathbb R^N$ satisfying $\mathscr H^{N-1}(\partial \Omega\setminus\Omega^{(0)})<\infty$ is $\mathcal{DM}^\infty$-extentable. The argument in \cite{CLT20} is based on the convolutions with the standard mollifiers.
		\end{remark}

\subsection{A version of the Gauss-Green formula on sets with one-sided approximations}

The key property of boundary-regular open sets is the existence of interior approximations in Lemma \ref{BR-lem-4.5}. In this subsection, we will extend the interior approximations to the general sets of finite perimeter.

We first consider the approximations for general $BV$ functions. Let $f\in BV(X)$. By the definition, there exists a sequence of $f_j\in Lip_{\rm loc}(X)$ such that
$$f_j\overset{L^1(X)}{\to } f\quad{\rm and}\quad  \int_X|\nabla f_j|{\rm d}\meas\to |Df|(X).$$
The problem with this approximation is that we do not have a pointwise estimate of $f_j$. We will recall the construction of approximating sequences of $f$ in \cite{KKST14, KLLS19}.

We need some notations. Let $f\in BV(X)$, the   lower and upper approximate limits of $f$ at $x$ are given by
$$f^{\wedge }(x):=\sup\left\{t\in\mathbb R \ \big|\ \lim_{r\to0}\frac{\meas(B_r(x)\cap \{f<t\})}{\meas(B_r(x))}=0\right\}$$
and
$$f^{\vee}(x):=\inf\left\{t\in\mathbb R \ \big|\ \lim_{r\to0}\frac{\meas(B_r(x)\cap \{f>t\})}{\meas(B_r(x))}=0\right\}.$$
The approximate jump set of $f$ is
$$S_f:=\{f^\wedge < f^\vee\}.$$

 We need the following approximations for $BV$ functions, given essentially in \cite{KKST14}.
\begin{proposition}[\cite{KKST14}]\label{BR-prop-4.10}
	Let $f\in BV(X)\cap L^\infty(X)$. Assume that $\{ f^\vee\not=0\}\cup \{f^\wedge\not=0\}$ is contained in a ball $B_R(z_0)$ with radius $R$.
	 Then there exist a sequence of functions $f_k\in Lip_{0}(B_{2R}(z_0))$, $k=1,2,\cdots,$ such that
	\begin{itemize}
		\item [(i)]  $f_k\to f$ in $L^1(X)$ as $k\to\infty$.
		\item [(ii)] For any $x_0\in X\setminus S_f$, $\lim_{k\to\infty}f_k(x_0)= f(x_0)$.
		\item [(iii)] Let $x_0\in S_f$. Suppose that there exists $t\in (0,1)$   such that  the set     $E_{t^*}:=\{-f>t^*\}$ is of locally finite perimeter,
and that
 $$\liminf_{r\to0}\frac{\meas\big(E_{t^*}\cap B_r(x_0)\big)}{\meas\big(B_r(x_0)\big)} \gs \gamma$$
  for some $\gamma\in(0,1)$, where
 $$t^*:=-\big(f^\wedge(x_0)-f^\vee(x_0)\big)\cdot t-f^\vee(x_0).$$
 Then it holds that
  $$\limsup_{k\to\infty}f_k(x_0)\ls \left(1-\frac{\gamma}{tc}\right)\cdot  f^\vee(x_0)+  \frac{\gamma}{tc}\cdot f^\wedge(x_0),$$
 where the constant $c$ depends only on $N,K,R$.
 	 		\item [(iv)] there exists $c:=c_{N,K,R}>0$ such that
 	 		$$\limsup_{k\to\infty}\int_X\eta |\nabla f_k|{\rm d}\meas\ls c\int_X\eta {\rm d} |Df|,\ \quad \forall \eta \in Lip_0(X), \ \eta\gs0.$$
 	 	\end{itemize}
	\end{proposition}
\begin{proof} The construction of this sequence $f_k$ has been given in \cite[Proposition 4.1]{KKST14}. For the completeness, we report their construction below.

Let $\varepsilon>0$ so that $\varepsilon<\frac{R}{100}$. Because $\meas$ is doubling on $B_{4R}(z_0)$, we can cover $B_{4R}(z_0)$ by a countable collection $\{ B_i\}$, $i=1,2,\cdots, $ of ball $B_i=B_\varepsilon(x_i)$ such that
$$\sum_{i=1}^\infty \chi_{20 B_{i}}\ls c,$$
where $20B_i:=B_{20\varepsilon}(x_i)$, and  $c$ depends only on the doubling constant on $B_{4R}(z_0)$ (and hence it depends only on $N,K$ and $R$). Remark that $(X,d,\meas)$ supports the locally weak (1,1)-Poincar\'e inequality (\ref{equ-add-Poincare-ineq}).
	Let $\varphi_i$ be a partition of unity with $0\ls \varphi_i\ls 1$, $\sum_{i=1}^\infty\varphi_i\equiv1$, $\varphi_i$ is $c/\varepsilon$-Lipschitz and $\varphi_i>0$ in $B_i$, and ${\rm supp}(\varphi_i)\subset 2B_i$ for each $i=1,2,\cdots.$ We set
	$$f_\varepsilon(x):=\sum_i^\infty f_{5B_i}\varphi_i(x),\quad x\in B_{4R}(z_0),$$
	and $f_\varepsilon(x)=0$ for any $x\not\in B_{4R}(z_0)$, where $f_{5B_i}=\fint_{B_{5\varepsilon}(x_i)}f{\rm d}\meas$. The function $f_\varepsilon$ is sometimes called the discrete convolution of $f$ (see, for example, \cite{KKST12, KLLS19}).
	It is clear that $f_\varepsilon\in Lip_0(B_{2R}(z_0))$ since $\{ f^\vee\not=0\}\cup \{f^\wedge\not=0\}$ is contained in $B_R(z_0)$ and $\varepsilon<R/100$.	

	For each $x\in B_j$, we have
	$$|f_\varepsilon(x)-f(x) |\ls \sum_{i: 2B_i\cap B_j\not=\emptyset}|f(x)-f_{5B_i}|.$$
	This implies that $f_\varepsilon\to f$ in $L^1(X)$ as $\varepsilon\to0$, and that
	$f_\varepsilon\to f$ in $X\setminus S_f$ by \cite[Theorem 3.4]{KKST14}. By taking $\varepsilon=1/k$, we conclude the items (i) and (ii).
	
	Fix any $x_0\in S_f$. Let $t$ and $\gamma$ satisfy the conditions in item (iii).  Letting $u:=-f$ and repeating the calculation in \cite[Page 60-61]{KKST14}, one can get
	$$\limsup_{\varepsilon\to0}f_\varepsilon(x_0) \ls \left(1-\frac{\gamma}{tc}\right)\cdot  f^\vee(x_0)+  \frac{\gamma}{tc}\cdot f^\wedge(x_0),$$	
	which is the desired estimate in item (iii) (by letting $\varepsilon=1/k$).
	
	From \cite[Equ. (4.1)]{KKST14}, we have
	$$|\nabla f_\varepsilon|\ls c\sum_{j=1}^\infty\frac{|Df|(10 B_j)}{\meas(B_j)}\cdot \chi_{B_j}.$$
	Thus, for any nonnegative $\eta\in Lip_0(X)$, we obtain
	\begin{equation*}
		\begin{split}
			\int_X\eta |\nabla f_\varepsilon|{\rm d}\meas&\ls c\sum_{j=1}^\infty\frac{|Df|(10 B_j)}{\meas(B_j)}\int_{B_j}\eta{\rm d}\meas=c\sum_{j=1}^\infty\eta_{B_j}\cdot |Df|(10 B_j)\\
			&=c \int_X\sum_{j=1}^\infty\eta_{B_j}\cdot\chi_{10B_j}{\rm d}|Df|=c\int_X\psi_\varepsilon{\rm d}|Df|,
		\end{split}
	\end{equation*}
	where
	$$\psi_\varepsilon:=\sum_{j=1}^\infty\eta_{B_j}\cdot\chi_{10B_j}.$$
	By the bounded overlap of the collection $10B_j$, $j=1,2,,\cdots,$ we have
	$$\psi_\varepsilon(x)=\sum_{j: 10B_j\ni x}\eta_{B_j}	\ls c\sum_{j: 10B_j\ni x}\frac{1}{\meas(B_{11\varepsilon}(x))}\int_{B_{11\varepsilon}(x)}\eta{\rm d}\meas\ls c'\fint_{B_{11\varepsilon}(x)}\eta {\rm d}\meas.$$
	As $\eta$ is continuous on each point $x\in X$, it follows
	$$\limsup_{\varepsilon\to0}\psi_\varepsilon(x)\ls c'\eta(x),\quad \forall x\in X.$$
	By the dominated convergence theorem, we conclude that
	$$\limsup_{\varepsilon\to0}\int_X\eta|\nabla f_\varepsilon|{\rm d}\meas\ls c\limsup_{\varepsilon\to0}\int_X\psi_\varepsilon{\rm d}|Df|\ls cc'\int_X\eta{\rm d}|Df|,$$	
	which is the desired estimate in the item (iv). The proof is finished.
	\end{proof}

We have the following approximation by using the above proposition to $f=\chi_\Omega$ for a bounded open set  $\Omega$ of finite perimeter.
\begin{lemma}\label{BR-lem-4.11}
	Let $\Omega$ be a bounded open set of finite perimeter. Assume that $$\liminf_{r\to0}\frac{\meas\big(B_r(x_0)\setminus\Omega \big)}{\meas\big(B_r(x_0)\big)} \gs \gamma,\quad \forall x_0\in \partial\Omega,$$
  for some $\gamma\in (0,1)$. 	
	 Then there is a sequence $\phi_k\in Lip_0(\Omega)$, $0\ls \phi_k\ls 1$, such that   $\phi_k(x)\to \chi_\Omega$ in $L^1(X)$, as $k\to\infty$,  and
	\begin{equation}
		\label{BR-equ-4.9}
	\limsup_{k\to\infty}\int_X\eta|\nabla \phi_k|{\rm d}\meas \ls \frac{c}{\gamma}\cdot\int_X\eta {\rm d}Per_{\Omega} , \quad \forall \eta\in Lip_0(X),
	 \end{equation}
	 	where the constant $c$ depends only on $N,K$ and ${\rm diam}(\Omega)$.	
\end{lemma}
\begin{proof}
	Let $f=\chi_\Omega$. For any $x_0\in S_f\subset \partial\Omega$, we have $f^\vee=1$ and $f^\wedge=0$.
	
 From the above Proposition \ref{BR-prop-4.10} (iii), by taking  $t=1/2$, we get $t^*=-1/2$ and then $E_{-1/2}=\{f<1/2\}=X\setminus\Omega$.
There exists a sequence $f_k\in Lip_0(B_{4R}(z_0))$, $R:={\rm diam}(\Omega),$ such that $0\ls f_k\ls 1$, $f_k\to f$ in $L^1(X)$,
 $$\limsup_{k\to\infty}f_k\ls 1-\frac{\gamma}{2c}\quad {\rm on}\ X\setminus\Omega,$$
and
$$\limsup_{k\to\infty}\int_X\eta|\nabla f_k|{\rm d}\meas \ls c\int_X\eta {\rm d} Per_{\Omega}, \quad \forall \eta\in Lip_0(X), \ \eta\gs0.$$

Since $\overline{B_{4R}(z_0)}\setminus\Omega$ is compact, from \cite[Lemma 6.1]{KLLS19}, there exists a convex combination of $f_k$, denoted by $\hat f_k$, such that
$$\hat f_k(x)\ls 1-\frac{\gamma}{2c}+\frac{\gamma}{4c}=1-\frac{\gamma}{4c},\quad \forall x\in \overline{B_{4R}(z_0)}\setminus\Omega,$$
for all $k$ sufficiently large. Remark that $f_k=0$ outside of $B_{4R}(x_0)$, so is $\hat f_k$. Hence $\hat f_k\ls 1-\gamma/(4c)$ for any $x\in X\setminus\Omega $ for all $k$ sufficiently large.  Because $f_k\to \chi_\Omega$ in $L^1(X)$, as $k\to\infty$. There exists a subsequence of $\{ f_k\}$, denoted by $\{ f_k\}$ again, such that
$$  f_k(x)\to \chi_\Omega(x), \quad \meas{\rm-a.e.} \ x\in X.$$
Therefore, the convex combinations
$$ \hat f_k(x)\to \chi_\Omega(x), \quad \meas{\rm-a.e.} \ x\in X.$$
As in \cite{KLLS19}, we put
$$\phi_k:=\frac{\max\{\hat f_k-1+\gamma/(4c)-1/k,\ -1/k\}}{\gamma/(4c)},\qquad \forall k\in\mathbb N.$$
Then $\phi_k(x)\to \chi_\Omega(x)$ for $ \meas{\rm-a.e.} \ x\in X$, as $k\to\infty$. The dominated convergence theorem ensures $\phi_k\to  \chi_\Omega$ in $L^1(X)$. The combination of  $|\nabla \phi_k|\ls (4c/\gamma)\cdot|\nabla \hat f_k|$ for each $k\in\mathbb N$ and the fact that $\hat f_k$ is a convex combination of $f_k$   yields
$$\limsup_{k\to \infty}\int_X\eta|\nabla\phi_k|{\rm d}\meas\ls (4c/\gamma)\int_X\eta {\rm d}Per_\Omega$$
for any nonnegative $\eta\in Lip_0(X)$.

At last, we check that $\phi_k\in Lip_0(\Omega)$ for all $k$ sufficiently large. Since $\phi_k$ is Lipschitz continuous on $X$, it suffices to check ${\rm supp}(\phi_k)\subset \Omega$ for all $k$ sufficiently large.  By $\hat f_k\ls 1-\gamma/(4c)$ for any $x\in X\setminus\Omega$ for all $k$ sufficiently large, we have
$$\phi_k=-1/(k\cdot \gamma/4c)<0 \quad {\rm on}\ X\setminus\Omega,$$
which implies ${\rm supp}(\phi_k)\subset \Omega$.
The proof is finished.
\end{proof}

\begin{remark}\label{BR-rem-4.12}
	(1) Recalling in the Euclidean setting $X=\mathbb R^N$, it is proved \cite[Theorem 1.8]{GHL23} that $\Omega$ can be approximated  by  smooth sets from the interior (that is, there exists a sequence of smooth domains $\{\Omega_j\}_{j\in \mathbb N}$ such that $\Omega_j\Subset\Omega_{j+1}$ for each $j$, $\Omega=\cup_j\Omega_j$ and $Per_{\Omega_j}(\mathbb R^N)\to Per_\Omega(\mathbb R^N)$) if and only if
		  $\mathscr H^{N-1}\big(\partial\Omega\cap \Omega^{(1)}\big)=0.$	
		
		  (2) It is clear that
$$\liminf_{r\to0}\frac{\meas\big(B_r(x_0)\setminus\Omega \big)}{\meas\big(B_r(x_0)\big)}>0,\quad \forall x_0\in \partial\Omega,$$
implies $\partial\Omega\cap \Omega^{(1)}=\emptyset.$

\end{remark}

From (\ref{BR-equ-4.9}), there exists a subsequence of $\phi_k$, denoted by $\phi_k$ again, such that
\begin{equation}\label{equation-4-9}
	|\nabla \phi_k|\cdot\meas \rightharpoonup \mu\quad {\rm weakly\ as }\ k\to \infty,
\end{equation}
  for some Radon measure $\mu$, and moreover,
$$\mu(U)\ls (c/\gamma)\cdot Per_\Omega(U),\quad  \forall U\subset X, \ U \ {\rm is\ open}.$$

\begin{proof}[Proof of Proposition \ref{BR-prop-1.12}]
This is given by repeating the proof of Proposition \ref{BR-prop-4.6}, and replacing  Lemma \ref{BR-lem-4.5} by the above measure $\mu$ in (\ref{equation-4-9}).
\end{proof}

\section{Boundary estimates for nonnegative subharmonic functions}

Let $(X,d,\meas)$ be an $RCD(K, N)$ space. Without loss of generality, we assume again that $K<0$ and $N>1$.
 We will show the following estimate for nonnegative subharmonic functions in Zygmund spaces:

\begin{theorem}\label{BR-thm-5.1-hf}
 Let $\Omega$ be as in Theorem \ref{BR-thm-1.8-main}.
 Let $f\in W^{1,2}(\Omega)\cap C(\overline\Omega)$ be a sub-harmonic function on $\Omega$ and $f\gs0$. Suppose that there exist  $x_0\in \partial \Omega$ and a constant $L>0$ such that
  $$f(y)\ls L\cdot d(y,x_0)\qquad \forall\ y\in \partial \Omega.$$ Then it holds
\begin{equation}
	\label{BR-equ-5.1}
	f(x)   \ls c_{1} L\cdot \delta(x) \ln\left(\frac{e\cdot {\rm diam}(\Omega)}{\delta(x)}\right)
	\end{equation}
for any $x\in\Omega$  such that $d(x,x_0)\ls 2\delta(x)=2d(x,\partial\Omega)$, where the constant $c_1$ depends only on $K$, $N$, $ {\rm diam}(\Omega), \meas(\Omega), Per_\Omega(X),  R_{\rm ext}$ and $C_0$ in  \eqref{BR-equ-1-3}.
\end{theorem}

We need the following simple fact.
\begin{lemma}\label{BR-lem-5.2}
	Let  $\Omega$ be a bounded domain. If it satisfies the exterior ball condition, then it holds $$\liminf_{r\to0}\frac{\meas\big(B_r(x_0)\setminus\Omega \big)}{\meas\big(B_r(x_0)\big)} \gs \gamma,\quad \forall x_0\in \partial\Omega,$$
	for some constant $\gamma$ depending only on $N, K, {\rm diam}(\Omega)$.
	\end{lemma}
	\begin{proof}
	For each $x_0\in \partial\Omega$, let   $B_{r_0}(y_0)$	 be an exterior ball with the center $y_0$ and the radius $r_0=d(y_0,\partial\Omega)=d(x_0,y_0)$.
	Then for any $r<r_0/3$, we have
	$$B_{3r}(x_0)\setminus\Omega\supset B_{r}(y_r),\quad B_{5r}(y_r)\supset B_{3r}(x_0),$$
	where $y_r$ is a point such that $d(x_0,y_r)=r$ and $d(y_r,y_0)=r_0-r$ (the existence of $y_r$ is ensured by the fact that $X$ is a geodesic space). Therefore, we get
	$$\meas\big(B_{3r}(x_0)\setminus\Omega\big)\gs \meas\big(B_r(y_r)\big)\gs \gamma\cdot \meas\big(B_{5r}(y_r)\big)\gs \gamma\cdot\meas\big(B_{3r}(x_0)\big)$$
	for any $r\ls r_0/3,$ where the constant $\gamma$ depends only on $N, K$ and ${\rm diam}(\Omega).$ This implies the desired assertion.	
	\end{proof}

To show Theorem  \ref{BR-thm-5.1-hf}, we first give the following estimate.
\begin{theorem}\label{BR-thm-5.3}
Let $\Omega$ be a bounded open set of finite perimeter, satisfying the uniformly exterior ball condition with radius $R_{\rm ext}$.
 Let $f\in W^{1,2}(\Omega)\cap C(\overline\Omega)$ be non-negative and satisfy ${\bf \Delta}f\gs -a$ for some $a(x)\in L^1(\Omega)$, $a\gs0$, in the sense of distriabutions. Then it holds
\begin{equation}
	\label{BR-equ-5.2}
	\begin{split}
			f(x)\ls & c_4\cdot \delta(x)  \int_{\partial\Omega}f(y)\left(  \int_{d^2(x,y)}^{D^2}\frac{{\rm d}t}{t\cdot \meas(B_{\sqrt t}(x))}+1 \right){\rm d}Per_\Omega(y)\\
			& +\int_\Omega a(y) G^\Omega(x,y) {\rm d}\meas(y)	
	\end{split}
	\end{equation}	
for any $x\in\Omega$, where the constant $D:=\max\{ {\rm diam}(\Omega), R_{\rm ext} \}$ and $c_4:=C_{N,K,{\rm diam}(\Omega)}\cdot c_3$ is  given in \eqref{BR-equ-3.13} and  \eqref{BR-equ-3.6}.
\end{theorem}

\begin{proof}
	Fix arbitrarily $\varepsilon>0$ and $x\in\Omega$ such that $\delta(x)> 2\sqrt\varepsilon.$ 	
	We define the function
	$$h_x(y):= c_4\delta(x)\left(  \int_{d^2(x,y)}^{D^2}\frac{{\rm d}t}{t\cdot \meas(B_{\sqrt t}(x))}+1 \right)$$	
	for any $y\in \overline{\Omega}\setminus\Omega_{\sqrt{\varepsilon}},$ where $D:=\max\{ {\rm diam}(\Omega), R_{\rm ext} \}$ and $c_4:=C_{N,K,{\rm diam}(\Omega)}\cdot c_3$ is  given in \eqref{BR-equ-3.13} and  \eqref{BR-equ-3.6}.
Since $x\in\Omega_{2\sqrt\varepsilon}$, it is clear that $h_x$ is continuous on $\overline{\Omega}\setminus\Omega_{\sqrt{\varepsilon}}.$
	
	We claim that there exists  $\tilde{h}\in C(\overline\Omega)$ such that  $|\nabla G_\varepsilon(x,\cdot)|(y)\ls \tilde h(y)$ for almost all $y\in\Omega$, and $\tilde h(y)=h_x(y)$ for all $y\in \overline{\Omega}\setminus\Omega_{\sqrt\varepsilon/2}$.	
	
	From Lemma \ref{BR-lem-3.10}, we have $  G_\varepsilon(x,\cdot)\in Lip_{\rm loc}(\Omega)$. There is a positive number $C_{x,\varepsilon}$ such that $|\nabla G_\varepsilon(x,\cdot )|(y)\ls C_{x,\varepsilon}$   for almost all $y\in\overline{\Omega_{\sqrt\varepsilon/2}}.$ Since $h_x$ is continuous on $\overline{\Omega_{\sqrt{\varepsilon}/2}}\setminus\Omega_{\sqrt{\varepsilon}}$, we can assume that $h_x\ls C'_{x,\varepsilon}$ for some $C'_{x,\varepsilon}>0$ on $\overline{\Omega_{\sqrt{\varepsilon}/2}}\setminus\Omega_{\sqrt{\varepsilon}}$. Let $C''_{x,\varepsilon}:=\max\{C_{x,\varepsilon},C'_{x,\varepsilon}\}.$ We define
	$$\tilde{h}(y)=
	\begin{cases}
		C''_{x,\varepsilon},& y\in \Omega_{\sqrt{\varepsilon}},\\
		\eta(y) h_x(y)+\big (1-\eta(y)\big)C''_{x,\varepsilon}, & y\in \overline{\Omega}\setminus\Omega_{\sqrt{\varepsilon}},
	\end{cases}$$
	where $\eta(y)$ is a contnuous function on $\overline\Omega$ such that $0\ls \eta\ls 1,$ $\eta(y)=0$ on $\Omega_{\sqrt{\varepsilon}}$, and $\eta(y)=1$ on $\overline{\Omega}\setminus \Omega_{\sqrt{\varepsilon}/2}$. It is easy to see that $\tilde h\in C(\overline\Omega)$. By using (\ref{BR-equ-3.13}) and  Theorem \ref{BR-thm-3.9}, and  noticing that $h_x(y)\ls \tilde h(y)$ on $\Omega\setminus\Omega_{\sqrt{\varepsilon}}$, we have $|\nabla G_\varepsilon(x,\cdot)|\ls \tilde h(\cdot)$ almost all on $\Omega\setminus\Omega_{\sqrt{\varepsilon}}$. Combining the definition of $C_{x,\varepsilon}$, we conclude that $|\nabla G_\varepsilon(x,\cdot)|(y)\ls \tilde h(y)$ for almost all $y\in \Omega$. Now the claim is proved since $\tilde h=h_x$ on $\overline{\Omega}\setminus\Omega_{\sqrt\varepsilon/2}.$
	
 Noticing that $G_\varepsilon(x,\cdot)\in W^{1,2}_0(\Omega)$ and $\Delta_\Omega G_\varepsilon(x,\cdot)=-p^\Omega_\varepsilon(x,\cdot)$ in the sense of distributions, by combining the above claim, we now know that the function $g(\cdot):=G_\varepsilon(x,\cdot)$ meets all of conditions in Proposition \ref{BR-prop-1.12}. So, by combining with Lemma \ref{BR-lem-5.2}, we have
	$$\int_\Omega f(y) p^\Omega_\varepsilon(x,y){\rm d}\meas(y)\ls \int_\Omega\ip{\nabla f}{\nabla g}{\rm d}\meas
	+c\int_{\partial \Omega}f(y)  h_x(y){\rm d}Per_\Omega(y).$$
Since $g\gs0$, $g\in W^{1.2}_0(\Omega)$, $f\in W^{1.2}(\Omega)$ and ${\bf \Delta}f\gs -a$ in the sense of distributions, we have
$$\int_\Omega\ip{\nabla f}{\nabla g}{\rm d}\meas=-\int_\Omega g {\rm d}{\bf \Delta}f\ls \int_\Omega a(y) G_\varepsilon(x,y) {\rm d}\meas(y)\ls\int_\Omega a(y) G^\Omega(x,y) {\rm d}\meas(y),$$
where we have used $G_\varepsilon(x,y)\ls G^\Omega(x,y)$. Therefore, we conclude that
\begin{equation}\label{BR-equ-5.3}
	\int_\Omega f(y) p^\Omega_\varepsilon(x,y){\rm d}\meas(y)\ls \int_\Omega a(y) G^\Omega(x,y) {\rm d}\meas(y)+c \int_{\partial \Omega}f(y)h_x(y){\rm d}Per_\Omega(y)
\end{equation}
holds for any $\varepsilon>0$ with $2\sqrt{\varepsilon}<\delta(x).$

Finally, it suffices to check that
$$f(x)=\lim_{\varepsilon \to0}\int_\Omega f(y)p^\Omega_\varepsilon(x,y){\rm d}\meas(y)=\lim_{\varepsilon\to0}H^\Omega_\varepsilon f(x).$$
In fact, it holds for any $r>0$ that
\begin{equation*}
	\begin{split}
		|H_\varepsilon^\Omega f(x)-f(x) |&\ls \int_\Omega p^\Omega_\varepsilon(x,y)|f(y)-f(x)|{\rm d}\meas(y)\\
		&\ls {\rm osc}_{B_r(x)}f\cdot   \int_{\overline{B_r(x)}} p^\Omega_\varepsilon(x,y) {\rm d}\meas(y)+ 2M \int_{\Omega\setminus B_r(x)} p^\Omega_\varepsilon(x,y){\rm d}\meas(y),		
	\end{split}
\end{equation*}
where ${\rm osc}_{E}f=\sup_{y_1,y_2\in E}|f(y_1)-f(y_2)|$ and $M=\sup_{y\in\Omega}|f|.$
By the upper bound (\ref{BR-equ-2.8}) and $\meas(B_{\sqrt\varepsilon}(x))\gs C\cdot \varepsilon^{N/2}$, we have
$$\int_{\Omega\setminus B_r(x)} p^\Omega_\varepsilon(x,y){\rm d}\meas(y)\ls
C\frac{e^{-\frac{r^2}{C\varepsilon}}}{\varepsilon^{N/2}} \cdot \meas(\Omega\setminus B_r(x))\to 0,\quad {\rm as}\ \varepsilon\to0^+.$$
 Therefore, we get
$$\lim_{\varepsilon\to0^+}|H^\Omega_\varepsilon f(x)-f(x) |\ls {\rm osc}_{B_r(x)}f,$$
since $\int_{\overline{B_r(x)}} p^\Omega_\varepsilon(x,y){\rm d}\meas(y)\ls1$ for any $\varepsilon>0$. By continuity of $f$ and the arbitrariness of $r>0$, we have $\lim_{\varepsilon\to0^+}|H^\Omega_\varepsilon f(x)-f(x)|=0$.
The proof is finished.	
	\end{proof}

Next, we estimate the first term of the right-hand side of (\ref{BR-equ-5.2}) assuming that $\Omega$ is perimetrically regular in Definition \ref{BR-def-1.5}.
\begin{lemma}\label{BR-lem-5.4}
 Let $(X,d,\meas)$ be an $RCD(K,N)$ for $K< 0$ and $N> 1$, and let $ \Omega\subset X$ be a domain of finite perimeter. Let $x_0\in\partial\Omega$. Suppose that
(\ref{BR-equ-1-3}) holds for some constant $C_0>0$. Then we have
	$$\int_{\partial\Omega} d(x_0,y) \int_{d^2(x,y)}^{{\rm diam}(\Omega)^2}\frac{{\rm d}t  }{t\cdot\meas(B_{\sqrt{t}}(x))}{\rm d}Per_\Omega(y)\ls
			C\ln\left(\frac{{\rm diam}(\Omega)}{\delta(x)}\right),$$
for any $x\in\Omega$ with $d(x,x_0)<2\delta(x)$.
\end{lemma}

\begin{proof}
For any $y\in\partial\Omega$, we have $d(x,y)\gs \delta(x)$.
The doubling property of $\meas$,  (\ref{BR-equ-2.4}),  implies for any $t\gs d^2(x,y)$ that
$$\meas(B_{\sqrt t}(x))\gs C\cdot \meas(B_{3\sqrt{t}}(x))\gs  C\cdot \meas(B_{3\sqrt{t}-2\delta(x)}(x_0))\gs C\cdot \meas(B_{\sqrt{t }}(x_0)),$$
where we used $2\sqrt t\gs 2d(x,y)\gs 2\delta(x).$ 	Then, by denoting $r(y):=d(x_0,y)$,
 \begin{equation*}
	\begin{split}
	&\quad \int_{\partial\Omega} r(y) \int_{d^2(x,y)}^{{\rm diam}(\Omega)^2}\frac{{\rm d}t}{t\cdot\meas(B_{\sqrt{t}}(x))}{\rm d}Per_\Omega(y)\\
	&\ls C\int_{\partial\Omega} \int_{d^2(x,y)}^{{\rm diam}(\Omega)^2}\frac{r(y)  }{t\cdot \meas(B_{\sqrt{t} }(x_0))}{\rm d}t{\rm d}Per_\Omega(y)\\
&= C
\int_{\delta^2(x)}^{{\rm diam}(\Omega)^2} \left(\int_{\partial\Omega\cap B_{ \sqrt{t}}(x)} r(y) {\rm d}Per_\Omega(y) \frac{1}{ t\cdot  \meas(B_{\sqrt{t}}(x_0))}\right) {\rm d}t\\
&\ls C'
\int_{\delta^2(x)}^{ {\rm diam}(\Omega)^2} \left(\int_{\partial\Omega\cap B_{3 \sqrt{t}}(x_0)} r(y) {\rm d}Per_\Omega(y) \frac{1}{ t\cdot  \meas(B_{\sqrt{t}}(x_0))}\right) {\rm d}t,
	\end{split}
\end{equation*}
where we have used $B_{\sqrt t}(x)\subset B_{\sqrt t+d(x_0,x)}(x_0)\subset B_{3\sqrt t}(x_0)$ for any $t\gs \delta^2(x)$.
By the assumption (\ref{BR-equ-1-3}), we have
$$\int_{\partial\Omega\cap B_{3 \sqrt{t}}(x_0)} r(y) {\rm d}Per_\Omega(y)\ls 3\sqrt t\cdot Per_\Omega(B_{3 \sqrt{t}}(x_0))\ls 3C_0\cdot \meas \big(B_{3 \sqrt{t}}(x_0)\big).$$
By (\ref{BR-equ-2.4}) again, it holds
\begin{equation*}
	\begin{split}
\int_{\partial\Omega} r(y) \int_{d^2(x,y)}^{{\rm diam}(\Omega)^2}\frac{{\rm d}t}{t\cdot\meas(B_{\sqrt{t}}(x))}{\rm d}Per_\Omega(y)& \ls C'
\int_{\delta^2(x)}^{ {\rm diam}(\Omega)^2}  \frac{3C_0\cdot\meas(B_{3\sqrt t}(x_0))}{ t\cdot  \meas(B_{\sqrt{t}}(x_0))}  {\rm d}t\\
&\ls C_1\cdot \int_{\delta^2(x)}^{ {\rm diam}(\Omega)^2} \frac{1}{ t} {\rm d}t,
	\end{split}
\end{equation*}
which completes the proof.
 \end{proof}

Now we are in the position to show the main result of this section.

\begin{proof}[Proof of Theorem \ref{BR-thm-5.1-hf}]
By combining Theorem \ref{BR-thm-5.3} and Lemma \ref{BR-lem-5.4}, we get
\begin{equation*}
\begin{split}
 f(x) & \ls c\cdot \delta(x)\cdot L\cdot\left(C\cdot\ln\left(\frac{{\rm diam}(\Omega)}{\delta(x)}\right)+\int_{\partial\Omega}d(y,x_0){\rm d}Per_\Omega(y)\right)\\
	&\ls c\cdot \delta(x)\cdot L\cdot\left(C\cdot\ln\left(\frac{{\rm diam}(\Omega)}{\delta(x)}\right)+ Per_\Omega(X)\cdot {\rm diam}(\Omega)\right)
	\end{split}
	\end{equation*}
	for any $x\in\Omega$ with $d(x,x_0)<2\delta(x)$. This finishes the proof.
\end{proof}

\section{Proof of Theorem \ref{BR-thm-1.8-main}}
In the last section, we will prove the main result --- Theorem \ref{BR-thm-1.8-main}.

We first recall the notion of $CAT(\kappa)$ space for any $\kappa\in\mathbb R$.   Let $(Y,d_Y)$ be a complete geodesic space, i.e., for any two points
$P_0,P_1\in Y$ there exists a point $S\in Y$, called a mid-point of $P_0$ and $P_1$, such that $$d_Y(P_0,S)=d_Y(S,P_1)=\frac{d_Y(P_0,P_1)}{2}.$$
Let $\kappa\in\mathbb R$. We say that a complete geodesic space $(Y,d_Y)$ is a $CAT(\kappa)$ space if for any three points $Q, P_0,P_1\in Y$ and any a mid-point $S$ of $P_0$ and $P_1$, there exist points  $\bar Q, \bar P_0, \bar P_1$ in the model space $\mathbb M^2_\kappa$ (2-dimensional simply-connected space form with sectional curvature $\kappa$) and the mid-point $\bar S$ of $\bar P_0$ and $\bar P_1$, such that
 $d_Y(Q,P_0)=\|\bar Q-\bar P_0\|_{\mathbb M^2_\kappa}$,  $d_Y(Q,P_1)=\|\bar Q-\bar P_1\|_{\mathbb M^2_\kappa}$,  $d_Y(P_0,P_1)=\|\bar P_0-\bar P_1\|_{\mathbb M^2_\kappa}$ and
 $$d_Y(Q,S)\ls\|\bar Q-\bar S\|_{\mathbb M^2_\kappa}.$$

Let $(X,d,\meas)$ be an $RCD(K,N)$ space with $N>1$ and $K\in\mathbb R$.
Suppose that  $\Omega\subset X$ is perimetrically regular and satisfies a uniformly exterior ball condition with radius $R_{\rm ext}\in(0,1)$.
Let $(Y, d_Y)$ be a $CAT(0)$ space.

  Given any $u_0\in Lip ({\overline\Omega},Y)$ with Lipschitz constant $L>0$,
 let $u\in W^{1,2}(\Omega,Y)$ be the harmonic map with boundary data  $u_0$ in Definition \ref{BR-def-1.2}.

\begin{proof}
	[Proof of Theorem \ref{BR-thm-1.8-main}]
Fix any $x\in \Omega$. Let $x^*\in \partial\Omega$ such that $d(x,x^*)=\delta(x).$
 From Theorem \ref{BR-thm-1.7}, we know that $u$ is continuous up to the boundary, and hence  it holds
$$d_Y\big(u(y),u(x^*)\big)=d_Y\big(u_0(y),u_0(x^*)\big)\leqslant L \cdot d(y,x^*),\quad \forall \ y\in\partial\Omega.$$
 On the other hand, we have that
 $${\bf \Delta} d_Y\big( u(\cdot),u(x^*)\big)\gs 0$$
 on $\Omega$ in the sense of distributions (see, \cite{MS22+, Gig23}).
  By using Theorem \ref{BR-thm-5.1-hf} to $f(\cdot):=d_Y\big(u(\cdot),u(x^*)\big)$, we conclude that
 \begin{equation}\label{BR-equ-6.1}
 	d_Y\big(u(z),u(x^*)\big)\ls c_1\cdot \delta(z)\cdot \ln\left(\frac{e\cdot {\rm diam}(\Omega)}{\delta(z)}\right)
 \end{equation}
 for any $z\in\Omega$ such that $d(z,x)\ls \frac{1}{10}\delta(x).$ 	 Let $R:=\frac{\delta(x)}{10}.$  For any $z\in B_R(x)$, it is clear, by triangle inequality, that
 $$\delta(z) =d(z,z^*)\gs d(x,z^*)-R\gs \delta(x)-R=\frac{9\delta(x)}{10},$$
 where $z^*\in\partial\Omega$ achieving $\delta(z)$, and
 $$d(z,x^*)\ls R+d(x,x^*)=\delta(x)+R=\frac{11\delta(x)}{10}.$$
 Therefore, we have $d(z,x^*)\ls 2\delta(z)$ and
 \begin{equation}
 	\label{BR-equ-6.2}
 	\frac{1}{2}\ls \frac{\delta(z)}{\delta(x)}\ls2
 \end{equation}
 for any $z\in B_R(x)$.
 By using the interior gradient estimate of $u$, established in \cite{MS22+,Gig23}, on the ball $B_{R}(x)\subset\Omega$, we have
 $$ \sup_{z',z''\in B_{R/2}(x),\ z'\not=z''}\frac{d_Y\big(u(z'), u(z'')\big)}{d(z',z'')}\ls  \frac{C_{N,K,{\rm diam}(\Omega)}}{R} \cdot \max_{z\in \overline{B_R(x)}} d_Y\big(u(z), u(x^*)\big).$$
 Combining with (\ref{BR-equ-6.1}) and (\ref{BR-equ-6.2}),  we conclude that
  \begin{equation*}
  	\begin{split}
  		 \sup_{z'\in B_{R/2}(x),\ z'\not=x}\frac{d_Y\big(u(z'), u(x)\big)}{d(z',x)}& \ls  \frac{C_{N,K,{\rm diam}(\Omega)}}{R} \cdot c_1\cdot \max_{z\in \overline{B_R(x)}}\left[\delta(z)\cdot\ln\left(\frac{e\cdot {\rm diam}(\Omega)}{\delta(z)}\right)\right]\\
  		 &\ls c_2\cdot \ln\left(\frac{2e\cdot {\rm diam}(\Omega)}{\delta(x)}\right).
  		 		   	\end{split}
  \end{equation*}
  It follows
  $$ {\rm Lip}u(x) \ls c_2\cdot \ln\left(\frac{2e\cdot {\rm diam}(\Omega)}{\delta(x)}\right). $$
 This is the desired result in Theorem \ref{BR-thm-1.8-main}. The proof is finished.
 \end{proof}

 \begin{proof}[Proof of Theorem \ref{BR-thm-1.2}] It is a direct consequence of Theorem \ref{BR-thm-1.8-main}, because the assumption that $\partial \Omega$ is Lipschitz continuous and bounded implies that it is perimetrically regular in Definition \ref{BR-def-1.5}(2).   \end{proof}

At last, we give a brief discussion on the case where the target $Y$ is a metric space with curvature bounded from above by a positive number  $\kappa$.

\begin{remark}\label{BR-remark-6.1}
	If $(Y,d_Y)$ is a $CAT(\kappa)$ for some $\kappa>0$ and let $u$ be a harmonic map with boundary data $u_0$ such that ${\rm diam}u_0(\Omega)<\frac{\pi}{2\sqrt{\kappa}}$, then  by an analogous argument as deriving (\ref{BR-equ-6.1}), one can obtian
	
	 \begin{equation}\label{BR-equ-6.3}
 	d_Y\big(u(z),u(x^*)\big)\ls c_2\cdot \delta(x)\ln\left(\frac{2e\cdot {\rm diam}(\Omega)}{\delta(x)}\right)
 \end{equation}
 for any $z\in\Omega$ such that $d(z,x)\ls \frac{1}{10}\delta(x).$ 	
	\end{remark}

\begin{remark}\label{BR-remark-6.2}
If $\Omega$ as in Theorem \ref{BR-thm-1.2} and $(Y,d_Y)$ be a $CAT(\kappa)$-space for some $\kappa>0$. Let $u$ be a harmonic map with Lipschitz data $u_0$ on $\partial\Omega$. Assume the image $u$ is contained in a ball of radius $<\pi/(2\sqrt\kappa)$. Then the assertion in Theorem \ref{BR-thm-1.2} still holds. This follows from (\ref{BR-equ-6.3}) and the gradient estimate in \cite{ZZZ19}.
  	\end{remark}

\end{document}